\documentclass[a4paper,reqno,11pt]{amsart}
\usepackage{amssymb,amstext,amsmath,amscd,amsthm,amsfonts,enumerate,graphicx,latexsym,mathrsfs}
\usepackage[usenames]{color}
\usepackage[all]{xy}
\makeatletter

\@addtoreset{equation}{section}
\makeatother
\newtheorem{theorem}{Theorem}[section]
\newtheorem*{theorem*}{Main Theorem}
\newtheorem{corollary}[theorem]{Corollary}
\newtheorem{lemma}[theorem]{Lemma}
\newtheorem{proposition}[theorem]{Proposition}

\newtheorem{definition-proposition}[theorem]{Definition-Proposition}
\theoremstyle{definition}
\newtheorem{definition}[theorem]{Definition}

\newtheorem*{question*}{Question}
\newtheorem{question**}{Question}

\newtheorem*{conjecture*}{Conjecture}
\newtheorem{example}[theorem]{Example}

\newtheorem{claim}{Claim}
\newtheorem*{claim*}{Claim}
\newtheorem{mainthm}{Theorem}

\def\Ext{\operatorname{Ext}}

\def\Hom{\operatorname{Hom}}
\def\Mod{\operatorname{\mathsf{Mod}}}\def\mod{\operatorname{\mathsf{mod}}}
\def\proj{\operatorname{\mathsf{proj}}}
\def\lex{\operatorname{\mathsf{lex}}}\def\Lex{\operatorname{\mathsf{Lex}}}
\def\dfct{\operatorname{\mathsf{def}}}\def\Dfct{\operatorname{\mathsf{Def}}}
\def\eff{\operatorname{\mathsf{eff}}}\def\Eff{\operatorname{\mathsf{Eff}}}

\def\fl{\operatorname{\mathsf{fl}}}
\def\Ab{\mathsf{Ab}}

\newcommand{\op}{\mathsf{op}}
\newcommand{\A}{\mathcal{A}}\newcommand{\B}{\mathcal{B}}
\newcommand{\C}{\mathcal{C}}

\renewcommand{\H}{\mathcal{H}}

\renewcommand{\P}{\mathcal{P}}
\renewcommand{\S}{\mathcal{S}}\newcommand{\T}{\mathcal{T}}
\newcommand{\U}{\mathcal{U}}\newcommand{\V}{\mathcal{V}}
\newcommand{\W}{\mathcal{W}}

\newcommand{\Ker}{\operatorname{Ker}}
\renewcommand{\Im}{\operatorname{Im}}
\newcommand{\Cok}{\operatorname{Cok}}

\newcommand{\id}{\mathsf{id}}

\newcommand{\xto}{\xrightarrow}
\begin{document}
\setlength{\baselineskip}{15pt}
\title[Auslander's defects over extriangulated categories]{Auslander's defects over extriangulated categories: an application for the General Heart Construction}
\author{Yasuaki Ogawa}
\address{Center for Educational Research of Science and Mathematics, Nara University of Education, Takabatake-cho, Nara, 630-8528, Japan}
\email{ogawa.yasuaki.gh@cc.nara-edu.ac.jp}
\maketitle

\begin{abstract}
The notion of extriangulated category was introduced by Nakaoka and Palu giving a simultaneous generalization of exact categories and triangulated categories.
Our first aim is to provide an extension to extriangulated categories of Auslander's formula:
for some extriangulated category $\C$, there exists a localization sequence $\dfct\C\to\mod\C\to\lex\C$,
where $\lex\C$ denotes the full subcategory of finitely presented left exact functors
and $\dfct\C$ the full subcategory of Auslander's defects.
Moreover we provide a connection between the above localization sequence and the Gabriel-Quillen embedding theorem.
As an application, we show that the general heart construction of a cotorsion pair $(\U,\V)$ in a triangulated category, which was provided by Abe and Nakaoka, is the same as the construction of a localization sequence $\dfct\U\to\mod\U\to\lex\U$.
\end{abstract}

%%%%%%%%%%%%%%%%%%%%%%%%%%%%%%%%%%%%%%%%%%%%%%%%%%%%%%%%
\section*{Introduction}
%%%%%%%%%%%%%%%%%%%%%%%%%%%%%%%%%%%%%%%%%%%%%%%%%%%%%%%%
Recently, the notion of extriangulated category was introduced in \cite{NP19} as a simultaneous generalization of exact categories and triangulated categories.
It allows us to unify many results on exact categories and triangulated categories in the same framework \cite{ZZ18, INP18, LN19}.
A typical example of extriangulated categories (which are possibly neither exact nor triangulated) is an extension-closed subcategory in a triangulated category.
Especially, the cotorsion class of a cotorsion pair in a triangulated category has a natural extriangulated structure.

Our first result is a further investigation of the following Auslander's result.
It was proved in \cite{Aus66} that, for any abelian category $\A$, the Yoneda embedding $\mathbb{Y}:\A\to\mod\A$ has an exact left adjoint $Q$, where $\mod\A$ is the category of finitely presented functors from $\A$ to the category of abelian groups.
Moreover the adjoint pair gives rise to a localization sequence
\begin{equation}\label{eq:Auslander's formula}
\xymatrix@C=1.2cm{
\dfct\A \ar[r]^-{{}}
&\mod\A\ar[r]^-{Q}\ar@/^1.2pc/[l]^-{}
&\A. \ar@/^1.2pc/[l]^{\mathbb{Y}}}
\end{equation}
Following \cite{Len98}, we call this \textit{Auslander's formula}.
Here $\dfct\A$ denotes the full subcategory of Auslander's defects in $\mod\A$ (see Definition \ref{def:defect}).
The first aim of this article is to present an extension to extriangulated categories of Auslander's formula:
for some extriangulated categories $\C$, there exists a localization sequence
\begin{equation}\label{eq:extriangulated Auslander's formula}
\xymatrix@C=1.2cm{
\dfct\C \ar[r]^-{{}}
&\mod\C\ar[r]^-{Q}\ar@/^1.2pc/[l]^-{}
&\lex\C\ar@/^1.2pc/[l]^{R}}
\end{equation}
where $\lex\C$ denotes the full subcategory of left exact functors in $\mod\C$ (Theorem \ref{thm:extri_Auslander's_formula}).
This localization sequence is closely related to the Gabriel-Quillen embedding theorem of exact categories (see Section \ref{sec:GQ}).
Furthermore, using the composed functor $E_\C:=Q\circ\mathbb{Y}:\C\to\mod\C\to\lex\C$,
we provide characterizations for the given extriangulated category $\C$ to be exact or abelian.

\begin{mainthm}[Theorem \ref{thm:exact_or_abelian}]\label{thm:A}
Let $\C$ be an extriangulated category with weak-kernels.
Then the following hold.
\begin{itemize}
\item[(1)] The functor $E_\C$ is exact and fully faithful if and only if $\C$ is an exact category.
\item[(2)] The functor $E_\C$ is an exact equivalence if and only if $\C$ is an abelian category. If this is the case, we have an equivalence $\C\simeq\lex\C$.
\end{itemize}
\end{mainthm}

Our second result is an application for a cotorsion pair $(\U,\V)$ in a triangulated category $\T$.
In \cite{Nak11, AN12}, it was proved that there exists an abelian category $\underline{\H}$ associated to the cotorsion pair, called the heart, and a cohomological functor $\mathbb{H}: \T\to\underline{\H}$.
This result has been shown for two extremal cases \cite{BBD, KZ08}, namely, t-structures and 2-cluster tilting subcategories (see \cite[Proposition 2.6]{Nak11} for details).
Since the cotorsion class $\U$ has a natural extriangulated structure, our first result shows the existence of the localization sequence $\dfct\U\to\mod\U\to\lex\U$.
Using this localization, we provide a good understanding for the heart and the cohomological functor.

\begin{mainthm}[Theorem \ref{thm:heart_is_lex}]\label{thm:B}
Let $(\U,\V)$ be a cotorsion pair in a triangulated category.
Then the heart $\underline{\H}$ of $(\U,\V)$ is naturally equivalent to $\lex\U$.
\end{mainthm}

The article is organized as follows:
In Section \ref{sec:preliminaries}, we deal with the definitions and basic properties concerning the Serre quotient and the extriangulated category.
In Section \ref{sec:defect}, we construct the localization sequence (\ref{eq:extriangulated Auslander's formula}) and prove Theorem \ref{thm:A}.
In Section \ref{sec:GQ}, we expand the sequence (\ref{eq:extriangulated Auslander's formula}) by taking direct colimits and explain how it relates to the Gabriel-Quillen embedding theorem.
Section \ref{sec:heart} is devoted to a proof of Theorem \ref{thm:B}.

\subsection*{Notation and convention}
All categories and functors appearing in this article are additive, unless otherwise specified.
The symbols $\A, \B, \C$ and etc. always denote additive categories, and the set of morphisms $X\rightarrow Y$ in $\C$ is denoted by $\C(X,Y)$ or simply denoted by $(X,Y)$ if there is no confusion.
If there exists a fully faithful functor $\U\hookrightarrow\C$,
we often regard $\U$ as a full subcategory of $\C$.
All subcategories in a given additive category are assumed to be full, additive and closed under isomorphisms.
We denote by $\C/[\U]$ the ideal quotient category of $\C$ modulo the (two-sided) ideal $[\U]$ in $\C$ consisting of all
morphisms having a factorization through an object in $\U$.
Consider a functor $F:\C\rightarrow \B$.
We define the \textit{image} and the \textit{kernel} of $F$ as the full subcategories
$$\Im F:=\{Y\in\B\mid {^\exists X\in\C},\  F(X)\cong Y\}\ \ \textnormal{and}\ \  \Ker F:=\{X\in\C\mid F(X)=0\},$$
respectively.
For a full subcategory $\U$ in $\C$, the symbol $F|_{\U}$ denotes the restriction of $F$ on $\U$.

Furthermore, we introduce the following notions:
For an additive category $\C$,
a \textit{(right) $\C$-module}  is defined to be a contravariant functor $\C \to \mathsf{Ab}$ and a \textit{morphism} $X\to Y$ between $\C$-modules $X$
and $Y$ is a natural transformation. Thus we define an abelian category $\Mod\C$ of $\C$-modules.
In the functor category $\Mod\C$, the morphism-space $(\Mod\C)(X,Y)$ is usually denoted by $\Hom_\C(X,Y)$.
We denote by $\mod\C$ the full subcategory of finitely presented $\C$-module in $\Mod\C$.
Let $\U$ be a full subcategory in $\C$.
We denote by $\mathsf{res}_\U:\Mod\C\to\Mod\U$ the restriction functor which sends $F$ to $F|_\U$.
We call the composed fuctor $\mathsf{res}_\U\circ\mathbb{Y}:\C\hookrightarrow\Mod\C\to\Mod\U$ the \textit{restricted Yoneda functor} of $\C$ relative to $\U$,
where $\mathbb{Y}$ is the usual Yoneda embedding.
We abbreviate as $\mathbb{Y}_\U:=\mathsf{res}_\U\circ\mathbb{Y}$.

%%%%%%%%%%%%%%%%%%%%%%%%%%%%%%%%%%%%%%%%%%%%%%%%%%%%%%%%
\section{Preliminaries}\label{sec:preliminaries}
%%%%%%%%%%%%%%%%%%%%%%%%%%%%%%%%%%%%%%%%%%%%%%%%%%%%%%%%
\subsection{The Serre quotient}
We firstly recall the definition and some basic properties of the localization theory of abelian categories.
Let $\A$ be an abelian category.
We recall that a full subcategory $\S$ in $\A$ is a \textit{Serre subcategory} if, for each exact sequence
$0\to X\to Y\to Z\to 0$ in $\A$, we have $Y\in\S$ if and only if $X, Z\in\S$.
In this case, we have a \textit{Serre quotient} $\A/\S$ of $\A$ relative to $\S$ which is known to be abelian.
We also recall that the natural quotient functor $Q:\A\to\A/\S$ is exact.
We denote this situation by the diagram
$\S\to\A\to\A/\S$ of functors.

\begin{proposition}\label{prop:right_adjoint1}
Let $\A$ be an abelian category and $\S$ its Serre subcategory.
If the quotient functor $Q:\A\to\A/\S$ admits a right adjoint $R$,
then the natural inclusion $\S\hookrightarrow\A$ also admits a right adjoint.
This situation will be denoted by the following diagram of functors
\begin{equation}\label{eq:localization_seq}
\xymatrix@C=1.2cm{
\S\ar[r]^-{{}}
&\A\ar[r]^-{Q}\ar@/^1.2pc/[l]^-{}
&\A/\S\ar@/^1.2pc/[l]^{R}}
\end{equation}
In this case, the Serre subcategory $\S$ is said to be \textnormal{localizing}, and we call this diagram a \textnormal{localization sequence} of $\A$ relative to $\S$.
\end{proposition}

Dually, we define the \textit{colocalizing subcategory} $\S$ and \textit{colocalization sequence} of $\A$ relative to $\S$.
If a given Serre quotient $\S\to\A\to\A/\S$ gives rise to a localizing and colocalizing sequence, we call this a \textit{recollement} of $\A$ relative to $\S$, and denote it as
\begin{equation}\label{eq:localization_seq}
\xymatrix@C=1.2cm{
\S\ar[r]^-{{}}
&\A\ar[r]^-{Q}\ar@/^1.2pc/[l]^-{}\ar@/^-1.2pc/[l]^-{}
&\A/\S\ar@/^1.2pc/[l]^{R}\ar@/^-1.2pc/[l]_{L}
}
\end{equation}
The following are standard examples of (co)localization sequences.

\begin{example}\label{ex:colocalization_of_functor_category}
\begin{itemize}
\item[(1)] Let $\C$ be an additive category with weak-kernels and $\P$ its contravariantly finite subcategory.
Then we have the colocalization sequence below
\begin{equation*}
\xymatrix@C=1.2cm{
\mod(\C/[\P])\ar[r]^-{{}}
&\mod\C\ar[r]^-{\mathsf{res}_\P}\ar@/^-1.2pc/[l]^-{}
&\mod\P\ar@/^-1.2pc/[l]_{L}
}
\end{equation*}
\item[(2)] Let $R$ be a Noetherian ring with an idempotent $e$.
Then, the sequence of canonical functors $\mod(R/ReR)\to\mod R\to\mod eRe$ gives rise to a recollement.
\end{itemize}
\end{example}

We should mention that the converse of Proposition \ref{prop:right_adjoint1} does not hold, that is,
even if the inclusion $\S\hookrightarrow\A$ admits a right adjoint, the quotient functor $Q:\A\to\A/\S$ does not necessarily 
admit a right adjoint.
The following gives a criterion for a Serre quotient to give rise to a localization sequence, e.g. \cite[Ch. 4. Thm. 4.5]{Pop}.

\begin{proposition}\label{prop:right_adjoint2}
Consider a Serre quotient $\S\to\A\to\A/\S$.
Then the following are equivalent:
\begin{itemize}
\item[(i)] Both the inclusion $\S\hookrightarrow\A$ and the quotient $\A\to\A/\S$ admit right adjoints;
\item[(ii)] For each $X\in\A$, there exists an exact sequence $S\to X\to Y$ in $\A$ satisfying $S\in\S$ and $\A(S', Y)=0=\Ext^1_\A(S', Y)$ for any $S'\in\S$.
\end{itemize}
\end{proposition}

In the case that a given category $\A$ is Grothendieck,
we have another criterion, e.g. \cite[Ch. 4. Prop. 6.3]{Pop}.

\begin{proposition}\label{prop:right_adjoint3}
Assume that $\A$ is a Grothendieck category.
For a Serre quotient $\S\to\A\to\A/\S$, the following are equivalent:
\begin{itemize}
\item[(i)] Both the inclusion $\S\hookrightarrow\A$ and the quotient $\A\to\A/\S$ admit right adjoints;
\item[(ii)] The Serre subcategory $\S$ is closed under coproducts.
\end{itemize}
\end{proposition}

The following subcategories associated to a given Serre subcategory $\S$ play important roles to understand the Serre quotient.

\begin{definition}
Let $\A$ be an abelian category and $\S$ its full subcategory.
We define the following full subcategories in $\A$.
\begin{itemize}
\item[(1)] Denote by $\S^{\perp_0}$ the full subcategory of objects $X$ satisfying $\A(S,X)=0$ for any $S\in\S$.
\item[(2)] Denote by $\S^{\perp_1}$ the full subcategory of objects $X$ satisfying $\Ext^1(S,X)=0$ for any $S\in\S$.
\item[(3)] We put $\S^\perp:=\S^{\perp_0}\cap\S^{\perp_1}$ which is called a \textit{perpendicular category of $\S$} in \cite{GL91}.
\end{itemize}
\end{definition}

The next lemma is a basic property of a localization sequence which will be freely used in many places.

\begin{lemma}\label{lem:property_of_localization_seq}
Let $\S\to\A\xto{Q}\A/\S$ be a Serre quotient.
\begin{itemize}
\item[(1)] The quotient functor $Q$ induces an isomorphism $\A(X, Y)\xto{\sim}(\A/\S)(QX, QY)$
for any $X\in\A$ and $Y\in\S^\perp$.
\item[(2)] The quotient functor $Q$ restricts a fully faithful functor $Q|_{\S^\perp}:\S^\perp\hookrightarrow\A/\S$. Moreover, if $Q$ has a right adjoint, then it is an equivalence.
\item[(3)] If $Q$ admits a right adjoint $R$, then $R$ is fully faithful and $\Im R=\S^\perp$ holds.
\end{itemize}
\end{lemma}

The following lemma says that a special adjoint pair gives rise to a localization sequence.

\begin{proposition}\cite[Ch. 4. Thm. 4.9]{Pop}
\label{prop:exact_functor_with_ff_adjoint}
Let $R:\C\to\A$ be a fully faithful functor between abelian categories.
If $R$ admits an exact left adjoint $Q$, then a sequence of canonical functors $\Ker Q\to\A\to\C$ gives rise to a localization sequence.
\end{proposition}

\subsection{Extriangulated categories}
We recall the definition and some needed properties of extriangulated categories from \cite{NP19}.
Throughout $\C$ denotes an additive category.
The symbol $\C^{\op}$ denotes the opposite category of $\C$.

\begin{definition}
Consider a biadditive functor $\mathbb{E}:\C^{\op}\times\C\to\Ab$.
For any objects $X, Z\in\C$, an element $\delta\in\mathbb{E}(X,Z)$ is called an \textit{$\mathbb{E}$-extension}.
Let $\delta'$ be an element in $\mathbb{E}(X',Z')$.
A \textit{morphism} $(z,x):\delta\to\delta'$ of $\mathbb{E}$-extensions is a pair of morphisms $x\in\C(X,X')$ and $z\in\C(Z,Z')$
with $z_*\delta =x^*\delta'$, where we set $z_*\delta:=\mathbb{E}(X,z)(\delta)$ and $x^*\delta':=\mathbb{E}(x,Z)(\delta')$.
\end{definition}

\begin{definition}
Let $\delta$ and $\delta'$ be the above $\mathbb{E}$-extensions.
By the biadditivity of $\mathbb{E}$, we have a natural isomorphism
$$\mathbb{E}(X\oplus X', Z\oplus Z')\cong \mathbb{E}(X, Z)\oplus \mathbb{E}(X, Z')\oplus \mathbb{E}(X', Z)\oplus \mathbb{E}(X', Z').$$
We denote by $\delta\oplus\delta'$ the element in $\mathbb{E}(X\oplus X', Z\oplus Z')$ corresponding to $(\delta, 0, 0, \delta')$ via this isomorphism.
\end{definition}

\begin{definition}
Let $X$ and $Z$ be objects in $\C$.
Two sequences of the form $Z\xto{g}Y\xto{f}X$
and $Z\xto{g'}Y'\xto{f'}X$ in $\C$
are said to be \textit{equivalent} if there exists an isomorphism $y: Y\to Y'$
which makes the following diagram commutative.
$$
\xymatrix@R=4pt@M=4pt{
&Y\ar[dd]^y_\cong\ar[rd]^f&\\
Z\ar[rd]_{g'}\ar[ru]^g&&X\\
&Y'\ar[ru]_{f'}&
}
$$
We denote the equivalence class of $Z\xto{g}Y\xto{f}X$
by $[Z\xto{g}Y\xto{f}X]$.
\end{definition}

\begin{definition}
Let $\C$ and $\mathbb{E}$ be as above.
\begin{itemize}
\item[(1)] For any $X, Z\in\C$, we denote as $0=[Z\xto{
\tiny\begin{bmatrix}
1\\
0
\end{bmatrix}
}Z\oplus X\xto{
\tiny\begin{bmatrix}
0\ 1
\end{bmatrix}
} X]$.
\item[(2)] For any two classes $[Z\xto{g}Y\xto{f}X]$ and $[Z'\xto{g'}Y'\xto{f'}X']$,
we denote by $[Z\xto{g}Y\xto{f}X]\oplus [Z'\xto{g'}Y'\xto{f'}X']$
the class $[Z\oplus Z'\xto{
g\oplus g'
}Y\oplus Y'\xto{
f\oplus f'
}X\oplus X']$.
\end{itemize}
\end{definition}

\begin{definition}
Let $\mathfrak{s}$ be a correspondence which associates an equivalence class $\mathfrak{s}(\delta)=[Z\xto{g}Y\xto{f}X]$
to any $\mathbb{E}$-extension $\delta\in\mathbb{E}(X,Z)$.
This $\mathfrak{s}$ is called a \textit{realization of $\mathbb{E}$},
if it satisfies the following condition
\begin{align*}\label{condition_realization}
  (*)\ \begin{cases}
    \intertext{Let $\delta\in \mathbb{E}(X,Z)$ and $\delta'\in\mathbb{E}(X', Z')$ be $\mathbb{E}$-extensions with
    $$\mathfrak{s}(\delta)=[Z\xto{g}Y\xto{f}X],\ \ \ \mathfrak{s}(\delta')=[Z'\xto{g'}Y'\xto{f'}X'].
    $$
    Then, for any morphism $(z, x):\delta\to\delta'$, there exists $y\in\C(Y, Y')$ which makes the following diagram commutes:
    $$
    \xymatrix@C=18pt@R=15pt{
    Z\ar[d]_z\ar[r]^{g}&Y\ar[d]^y\ar[r]^{f}&X\ar[d]^x\\
    Z'\ar[r]^{g'}&Y'\ar[r]^{f'}&X'
    }
    $$}
  \end{cases}
\end{align*}
Under the condition $(*)$, we say that the sequence $Z\xto{g}Y\xto{f}X$ \textit{realizes} $\delta$
and the triple $(z, y, x)$ \textit{relaizes} $(z, x)$.
\end{definition}

\begin{definition}
A realization $\mathfrak{s}$ of $\mathbb{E}$ is said to be \textit{additive},
if it satisfies the following conditions:
\begin{itemize}
\item[(i)] For any $X,Z\in\C$, the $\mathbb{E}$-extension $0\in\mathbb{E}(X, Z)$ satisfies
$\mathfrak{s}(0)=0$;
\item[(ii)] For any pair of $\mathbb{E}$-extensions $\delta$ and $\delta'$,
$\mathfrak{s}(\delta\oplus\delta')=\mathfrak{s}(\delta)\oplus\mathfrak{s}(\delta')$
holds.
\end{itemize}
\end{definition}

\begin{definition}
The triple $(\C,\mathbb{E},\mathfrak{s})$ is called an \textit{extriangulated category} if the following conditions are satisfied:
\begin{itemize}
\item[(ET1)] $\mathbb{E}: \C^{\op}\times\C\to\mathsf{Ab}$ is an additive bifunctor;
\item[(ET2)] $\mathfrak{s}$ is an additive realization of $\mathbb{E}$;
\item[(ET3)] Let $\delta\in\mathbb{E}(X,Z)$ and $\delta'\in\mathbb{E}(X',Z')$ be $\mathbb{E}$-extensions realized as
$$\mathfrak{s}(\delta)=[Z\xto{g}Y\xto{f}X],\ \ \ \ 
\mathfrak{s}(\delta')=[Z'\xto{g'}Y'\xto{f'}X'].$$
For any commutative square
$$
\xymatrix@C=18pt@R=15pt{
Z\ar[r]^{g}\ar[d]_z&Y\ar[r]^{f}\ar[d]^y&X\\
Z'\ar[r]^{g'}&Y'\ar[r]^{f'}&X'
}
$$
in $\C$,
there exists a morphism $(z, x):\delta\to\delta'$ satisfying $xf=f'y$.
\item[(ET3)$^{\op}$] Dual of (ET3).
\item[(ET4)] Let $\delta\in\mathbb{E}(X, Z)$ and $\delta'\in\mathbb{E}(A ,Y)$ be $\mathbb{E}$-extensions realized by
$$\mathfrak{s}(\delta)=[Z\xto{g}Y\xto{f}X],\ \ \ \ 
\mathfrak{s}(\delta')=[Y\xto{b}B\xto{a}A].$$
Then there exist an object $C\in\C$, a commutative diagram
$$
\xymatrix@C=18pt@R=15pt{
Z\ar@{=}[d]\ar[r]^g&Y\ar[r]^f\ar[d]^b&X\ar[d]^{b'}\\
Z\ar[r]^{g'}&B\ar[d]^a\ar[r]^{f'}&C\ar[d]^{a'}\\
&A\ar@{=}[r]&A
}
$$
in $\C$, and an $\mathbb{E}$-extension $\delta''\in\mathbb{E}(C, Z)$ realized by $Z\xto{g'}B\xto{f'}C$,
which satisfy the following compatibilities:
\begin{itemize}
\item[(1)] $X\xto{b'}C\xto{a'}A$ realizes $f_*\delta'$;
\item[(2)] $b'^*\delta''=\delta$;
\item[(3)] $g_*\delta''=a'^*\delta'$.
\end{itemize}
\end{itemize}
\item[(ET4)$^{\op}$] Dual of (ET4).
\end{definition}

In the rest, the symbol $(\C, \mathbb{E}, \mathfrak{s})$ denotes an extriangulated category.
We also write $\C:=(\C, \mathbb{E}, \mathfrak{s})$, if there is no confusion.

\begin{definition}
Let $(\C, \mathbb{E}, \mathfrak{s})$ be an extriangulated category.
\begin{itemize}
\item[(1)] A sequence $Z\xto{g}Y\xto{f}X$ is called a \textit{conflation} if it realizes some $\mathbb{E}$-extension $\delta\in\mathbb{E}(X,Z)$.
In this case, we denote the pair $(Z\xto{g}Y\xto{f}X, \delta)$ by $Z\overset{g}{\longrightarrow} Y\overset{f}{\longrightarrow} X\overset{\delta}{\dashrightarrow}$, which is called an \textit{$\mathbb{E}$-triangle}.
\item[(2)] A morphism $g\in\C(Z,Y)$ is called an \textit{inflation} if it admits some conflation $Z\xto{g}Y\to X$.
\item[(3)] A morphism $f\in\C(Y,X)$ is called an \textit{deflation} if it admits some conflation $Z\xto{}Y\xto{f} X$.
\end{itemize}
\end{definition}

Like the case of exact categories, the following hold for any extriangulated category:
The inflations and deflations are closed under composition by (ET4) and (ET4)$^{\op}$, respectively;
The finite coproduct of conflations is also a conflation by the additivity of $\mathbb{E}$ and $\mathfrak{s}$.

Next, we define the pullback (resp. pushout) as follows:
Let $\delta$ be an $\mathbb{E}$-extension with a realization $\mathfrak{s}(\delta)=[Z\xto{g}Y\xto{f}X]$ in $\C$.
For each morphism $x:X'\to X$, we get an $\mathbb{E}$-extension $x^*\delta$ with a realization $\mathfrak{s}(x^*\delta)=[Z\to E\xto{f'} X']$.
Then we have a morphism $(\id_Z, x)$ of $\mathbb{E}$-extensions.
Since $\mathfrak{s}$ is additive, there exists a commutative diagram
$$
\xymatrix@C=18pt@R=18pt{
Z\ar[r]\ar@{=}[d]&E\ar@{}[rd]|{(\textnormal{Pb})}\ar[d]_y\ar[r]^{f'}&X'\ar[d]^x\\
Z\ar[r]&Y\ar[r]^f&X
}
$$
which realizes $(\id_Z, x)$.
The commutative square $(\textnormal{Pb})$ is called a \textit{pullback} of a deflation $f$ along $x$.
Dually, a morphism $z:Z\to Z'$ induces a commutative diagram:
$$
\xymatrix@C=18pt@R=18pt{
Z\ar[d]_z\ar[r]^g\ar@{}[rd]|{(\textnormal{Po})}&Y\ar[d]^{y'}\ar[r]&X\ar@{=}[d]\\
Z'\ar[r]^{g'}&E'\ar[r]&X
}
$$
The commutative square $(\textnormal{Po})$ is called a \textit{pushout} of a inflation $g$ along $z$.

\begin{lemma}\label{lem:pullback}
\cite[Cor. 3.16]{NP19}
For the above pullback $(\textnormal{Pb})$ and pushout $(\textnormal{Po})$, we have $\mathbb{E}$-triangles
\begin{equation*}
E\overset{
\tiny\begin{bmatrix}
f'\\
y
\end{bmatrix}
}{\longrightarrow} X'\oplus Y\overset{
\tiny\begin{bmatrix}
x\ -f
\end{bmatrix}
}{\longrightarrow} X\dashrightarrow \ \ \textnormal{and}\ \ \ 
Z\overset{
\tiny\begin{bmatrix}
g\\
z
\end{bmatrix}
}{\longrightarrow} Y\oplus Z'\overset{
\tiny\begin{bmatrix}
y'\ -g'
\end{bmatrix}
}{\longrightarrow} E'\dashrightarrow, 
\end{equation*}
respectively.
\end{lemma}

Via the Yoneda lemma, any $\mathbb{E}$-extension $\delta\in\mathbb{E}(X, Z)$ corresponds to a morphism
$\delta_\sharp:\C(,X)\to\mathbb{E}(-,Z)$.

\begin{lemma}\label{lem:ext_seq}
\cite[Cor. 3.12]{NP19}
Let $(\C, \mathbb{E}, \mathfrak{s})$ be an extriangulated category.
For any $\mathbb{E}$-triangle $Z\overset{g}{\longrightarrow} Y\overset{f}{\longrightarrow} X\overset{\delta}{\dashrightarrow}$,
we have an exact sequence
$$\C(-,Z)\xto{g\circ -}\C(-,Y)\xto{f\circ -}\C(-,X)\xto{\delta_\sharp}\mathbb{E}(-,Z)\xto{g\circ -}\mathbb{E}(-,Y)\xto{f\circ -}\mathbb{E}(-,X)$$
in $\Mod\C$.
\end{lemma}

The extriangulated category is a simultaneous generalization of exact categories and triangulated categories.
We shall use the following terminology.

\begin{example}
An exact (resp. triangulated) structure in an additive category $\C$ can give rise to an extriangulated structure.
In this case, we say that the extriangulated category is \textit{exact} (resp. \textit{triangulated}) (see \cite[Prop. 3.22, Example 2.13]{NP19} for details).
\end{example}

We end this section by recalling the following fact for later use.

\begin{proposition}\cite[Cor. 3.18]{NP19}\label{prop:extri_to_exact}
Let $(\C, \mathbb{E}, \mathfrak{s})$ be an extriangulated category, in which any inflation is a monomorphism and any deflation is an epimorphism.
If we let $\mathbb{F}$ be the class of conflations given by the $\mathbb{E}$-triangles, then $(\C,\mathbb{F})$ is an exact category.
\end{proposition}

%%%%%%%%%%%%%%%%%%%%%%%%%%%%%%%%%%%%%%%%%%%%%%%%%%%%%%%%
\section{Auslander's defects over extriangulated categories}\label{sec:defect}
%%%%%%%%%%%%%%%%%%%%%%%%%%%%%%%%%%%%%%%%%%%%%%%%%%%%%%%%
\subsection{The Serre subcategory of defects}
We firstly recall that, for an additive category $\C$, although the subcategory $\mod\C$ is closed under cokernels and extensions in $\Mod\C$, it is not always abelian since it is not necessarily closed under kernels.
In fact, the following lemma is well-known.

\begin{lemma}\label{lem:weak-kernel}\cite[Thm. 1.4]{Fre}
The following are equivalent for an additive category $\C$:
\begin{itemize}
\item[(i)] The category $\C$ admits weak-kernels;
\item[(ii)] The full subcategory $\mod\C$ is an exact abelian subcategory in $\Mod\C$,
that is, it is abelian and the canonical inclusion $\mod\C\hookrightarrow\Mod\C$ is exact.
\end{itemize}
\end{lemma}

This subsection is devoted to studying basic properties of \textit{effaceable} functors and Auslander's \textit{defects} in $\mod\C$.
So, we assume that \textit{an extriangulated category $\C:=(\C, \mathbb{E}, \mathfrak{s})$ has weak-kernels}.
We shall prove that effaceable functors are nothing other than Auslander's defects and they form a Serre subcategory in $\mod\C$.
To begin, we study the subcategory $\eff\C$ of effaceable functors in $\mod\C$ defined as below (see \cite[p. 30]{Kel90}, \cite[p. 141]{Gr57}):

\begin{definition}\label{def:effaceable}
An object $F\in\mod\C$ is said to be \textit{effaceable} if it satisfies the following condition:
\begin{eqnarray}\label{condition_eff}
  \begin{cases}
    \textnormal{For any}\ X\in\C \ \textnormal{and any}\ x\in F(X),& \\
    \textnormal{there exists a deflation}\ \alpha: Y\twoheadrightarrow X\ \textnormal{such that}\ F(\alpha)(x)=0.
 &
  \end{cases}
\end{eqnarray}
We denote by $\eff\C$ the full subcategory of all effaceable functors in $\mod\C$.
\end{definition}

\begin{lemma}\label{lem:Serre_subcategory}
Let $\C$ be an extriangulated category with weak-kernels.
Then the subcategory $\eff\C$ is a Serre subcategory in $\mod\C$.
\end{lemma}
\begin{proof}
To show that $\eff\C$ is closed under extensions, 
let $0\to S_2\xto{f} F\xto{g} S_1\to 0$ be an exact sequence in $\mod\C$ with $S_1, S_2\in\eff\C$.
Consider $X\in\C$ and $x\in F(X)$.
Since $S_1\in\eff\C$, for $x_1:=(g_X)(x)$, there exists a deflation $\alpha_1:Y_1\to X$ with $S_1(\alpha_1)(x_1)=0$.
It forces that $x':=F(\alpha_1)(x)$ satisfies $(g_{Y_1})(x')=0$.
Hence $x'$ belongs to $S_2(Y_1)$.
Since $S_2\in\eff\C$, for $x'\in S_2(Y_1)$, there exists a deflation $\alpha_2:Y_2\to Y_1$ with $S_2(\alpha_2)(x')=0$.
It is easily checked that $F(\alpha_2\alpha_1)(x)=0$.
The observation here can be understood by chasing the following commutative diagram with exact rows
$$
\xymatrix@C=18pt@R=18pt{
0\ar[r]&S_2(X)\ar[r]\ar[d]&F(X)\ar[r]^{g_X}\ar[d]_{F(\alpha_1)}&S_1(X)\ar[r]\ar[d]^{S_1(\alpha_1)}&0\\
0\ar[r]&S_2(Y_1)\ar[r]\ar[d]_{S_2(\alpha_2)}&F(Y_1)\ar[r]^{g_{Y_1}}\ar[d]&S_1(Y_1)\ar[r]\ar[d]&0\\
0\ar[r]&S_2(Y_2)\ar[r]&F(Y_2)\ar[r]&S_1(Y_2)\ar[r]&0
}
$$
Since $\alpha_2\alpha_1$ is a deflation and $F(\alpha_2\alpha_1)(x)=0$, we have $F\in\eff\C$.

By a similar argument, it is easily checked that $\eff\C$ is closed under taking factors and subobjects.
\end{proof}

Next we define defects over extriangulated categories and provide a precise connection to effaceable functors.

\begin{definition}\label{def:defect}
Let $Z\overset{}{\longrightarrow} Y\overset{}{\longrightarrow} X\overset{\delta}{\dashrightarrow}$ be an $\mathbb{E}$-triangle in an extriangulated category $\C$.
Then we have an exact sequence $(-,Z)\to (-,Y)\to (-,X)\to \tilde\delta\to 0$ in $\mod\C$.
The functor $\tilde\delta$ is called a \textit{defect of $\delta$}.
We denote by $\dfct\C$ the full subcategory in $\mod\C$ consisting of all functors isomorphic to defects.
\end{definition}

Note that, by Lemma \ref{lem:ext_seq}, the defect $\tilde\delta$ is isomorphic to $\Im\delta_\sharp$.
The following proposition shows that defects are nothing other than effaceable functors.

\begin{proposition}\label{prop:char_of_Defect}
Let $\C$ be an extriangulated category with weak-kernels.
Then we have an equality $\eff\C=\dfct\C$.
In particular, $\dfct\C$ is a Serre subcategory in $\mod\C$.
\end{proposition}

To prove this proposition, we shall use the following lemma.

\begin{lemma}\label{lem:closed_under_cokernels}
The subcategory $\dfct\C$ is closed under taking kernels and cokernels in $\mod\C$.
\end{lemma}
\begin{proof}
Let $\delta$ and $\delta'$ be $\mathbb{E}$-extensions with realizations
$$\mathfrak{s}(\delta)=[Z\xto{g}Y\xto{f}X]\ \ \textnormal{and}\ \ \mathfrak{s}(\delta')=[Z'\xto{g'}Y'\xto{f'}X'].$$
Consider a morphism $\alpha: \tilde\delta'\to\tilde\delta$ in $\dfct\C$.
We shall show that $\Cok\alpha$ still belongs to $\dfct\C$.
By the Yoneda lemma, the morphism $\alpha$ induces a morphism between presentations of $\tilde\delta$ and $\tilde\delta'$, and hence we have the following commutative diagram
$$
\xymatrix@C=18pt@R=15pt{
Z'\ar[r]^{g'}&Y'\ar[r]^{f'}\ar[d]^y&X'\ar[d]^x\\
Z\ar[r]^g&Y\ar[r]^f&X
}
$$
in $\C$.
By taking the pullback $(\textnormal{Pb})$ of the deflation $f$ along $x$,
due to Lemma \ref{lem:pullback},
we get a conflation $\delta'':E\xto{\tiny\begin{bmatrix}
-y_1\\
a
\end{bmatrix}} Y\oplus X'\xto{[f\ x]} X$.
Hence, we have a morphism $y_2:Y'\to E$ which makes the following diagram commutative:
$$
\xymatrix{
Y'\ar@/^15pt/[rrd]^{f'}\ar@/_15pt/[rdd]_y\ar@{..>}[rd]|{y_2}&&\\
&E\ar[r]^a\ar[d]_{y_1}\ar@{}[rd]|{(\textnormal{Pb})}&X'\ar[d]^x\\
&Y\ar[r]^f&X
}
$$
Moreover, we construct the following commutative diagram
$$
\xymatrix@C=18pt@R=18pt{
Z'\ar[r]^{g'}&Y'\ar[r]^{f'}\ar[d]_{y_2}&X'\ar@{=}[d]\\
Z\ar[r]^b\ar@{=}[d]&E\ar[r]^a\ar[d]_{y_1}\ar@{}[rd]|{(\textnormal{Pb})}&X'\ar[d]^x\\
Z\ar[r]^g\ar[d]_{-b}&Y\ar[r]^f\ar[d]^{
\tiny\begin{bmatrix}
1\\
0
\end{bmatrix}
}
&X\ar@{=}[d]\\
E\ar[r]_{
\tiny\begin{bmatrix}
-y_1\\
a
\end{bmatrix}\ \ 
}&Y\oplus X'\ar[r]_{\ \ [f\ x]}&X
}
$$
with the rows being conflations, which induces the following commutative diagram
$$
\xymatrix@C=18pt@R=18pt{
(-,Z')\ar[r]^{(-,g')}&(-,Y')\ar[r]^{(-,f')}\ar[d]_{(-,y)}&(-,X')\ar[d]^{(-,x)}\ar[r]^p&\tilde\delta'\ar[r]\ar[d]^\alpha&0\\
(-,Z)\ar[r]^{(-,g)}\ar[d]_{(-,-b)}&(-,Y)\ar[r]^{(-,f)}\ar[d]_{(-,
\tiny\begin{bmatrix}
1\\
0
\end{bmatrix})
}
&(-,X)\ar@{=}[d]\ar[r]&\tilde\delta\ar[r]\ar[d]&0\\
(-,E)\ar[r]_{}&(-,Y\oplus X')\ar[r]_{}&(-,X)\ar[r]&\tilde\delta''\ar[r]\ar[d]&0\\
&&&0&
}
$$
in $\mod\C$ with exact rows.
As applications of the snake lemma, we obtain an exact sequence
$(-,X')\xto{\alpha p}\tilde\delta\to\tilde\delta''\to 0$.
Since $p$ is an epimorphism, we have $\Cok\alpha\cong\tilde\delta''$.
This shows that $\dfct\C$ is closed under taking cokernels.
Dually, it can be checked that $\dfct\C$ is closed under taking kernels.
\end{proof}

We remark that a similar proof appears in \cite[Thm. A.2]{Eno18} for exact categories.
Now we are in position to prove Proposition \ref{prop:char_of_Defect}.

\begin{proof}[Proof of Proposition \ref{prop:char_of_Defect}]
To show $\eff\C\subseteq\dfct\C$, let $S$ be an object in $\eff\C$.
Since $S\in\mod\C$, there exists an epimorphism $(-,X)\xto{x}S\to 0$.
Thanks to the Yoneda Lemma, we regard $x$ as an element in $S(X)$.
Since $S$ satisfies the condition (\ref{condition_eff}), we get a conflation $Z\to Y\xto{f} X$ such that $S(f)(x)=0$.
Let $\delta$ be an $\mathbb{E}$-extension with a realization $\mathfrak{s}(\delta)=[Z\to Y\to X]$.
We consider the defect of $\delta$, namely, there exists an exact sequence
$(-,Z)\to (-,Y)\xto{(-,f)} (-,X)\to \tilde\delta\to 0$ in $\mod\C$.
Since $S(f)(x)=0$ is equivalent to that the composed morphism $(-,Y)\xto{(-,f)}(-,X)\xto{x}S$ is zero,
we have thus obtained the following commutative diagram
$$
\xymatrix@C=18pt@R=15pt{
(-,Z)\ar[r]&(-,Y)\ar[rd]_0\ar[r]^{(-,f)}&(-,X)\ar@{->>}[d]^x\ar@{->>}[r]&\tilde\delta\ar[r]\ar@{.>>}[dl]&0\\
&&S&&
}
$$
By the commutativity, the above dotted arrow from $\tilde\delta$ to $S$ is an epimorphism.
Thus we have an exact sequence $0\to K\to\tilde\delta\to S\to 0$ in $\mod\C$.
Since $\eff\C$ is a Serre subcategory, we get $K\in\eff\C$.
By a similar argument to the above, there exists a defect $\tilde\delta'$ and an epimorphism $\tilde\delta'\to K\to 0$.
This induces an exact sequence $\tilde\delta'\to\tilde\delta\to S\to 0$.
Thanks to Lemma \ref{lem:closed_under_cokernels}, we conclude $S\in\dfct\C$.

To show the converse $\eff\C\supseteq\dfct\C$, let $\delta$ be an $\mathbb{E}$-extension with a realization $\mathfrak{s}(\delta)= [X''\xto{g}X'\xto{f} X]$ and consider its defect $\tilde\delta$, that is, there exists an exact sequence
$$(-,X'')\to (-,X')\to (-,X)\xto{p} \tilde\delta\to 0$$
in $\mod\C$.
To show $\tilde\delta\in\eff\C$, fix an object $Y\in\C$ and an element $y\in \tilde\delta(Y)$.
Since $p_Y:(Y,X)\to \tilde\delta(Y)$ is an epimorphism, we get $h\in (Y,X)$ such that $p_Y(h)=y$.
By taking the pullback of $f$ along $h$, we have the following commutative diagram with the rows being conflations:
$$
\xymatrix@C=18pt@R=18pt{
X''\ar[r]^g&X'\ar@{}[rd]|{(\textnormal{Pb})}\ar[r]^f&X\\
X''\ar[r]\ar@{=}[u]&Y'\ar[r]^{\alpha}\ar[u]&Y\ar[u]_h
}
$$
By using the Yoneda lemma, we can check that $\tilde\delta(\alpha)(y)=0$,
which says $\tilde\delta\in\eff\C$.
We have thus obtained $\eff\C=\dfct\C$.
\end{proof}

By the discussion so far, we get a Serre subcategory $\eff\C=\dfct\C$ in $\mod\C$.
Thus we have a Serre quotient of $\mod\C$ relative to $\dfct\C$.
Since it is basic to study the perpendicular category $(\dfct\C)^\perp$ to understand the Serre quotient,
we shall show that it coincides with the subcategory of left exact functors in $\mod\C$.

\begin{definition}
Let $\A$ and $(\C,\mathbb{E},\mathfrak{s})$ be an abelian category and an extriangulated category, respectively.
A contravariant functor $F:\C\to\A$ is said to be
\begin{itemize}
\item[(1)] \textit{half-exact}, if $F$ sends a conflation $X\to Y\to Z$ to an exact sequence $FZ\to FY\to FX$;
\item[(2)] \textit{left exact}, if $F$ is half-exact and sends a deflation $Y\to Z$ to a monomorphism $FZ\to FY$;
\item[(3)] \textit{right exact}, if $F$ is half-exact and sends an inflation $X\to Y$ to an epimorphism $FY\to FX$;
\item[(4)] \textit{exact}, if $F$ is left exact and right exact.
\end{itemize}
We denote by $\lex\C$ (resp. $\Lex\C$) the full subcategory of all left exact functors in $\mod\C$ (resp. $\Mod\C$).
\end{definition}

Let us remark that, if $\C$ is a triangulated category, the left exact functors should be zero.
In fact, any morphism in $\C$ is an inflation as well as a deflation.
So, a left exact functor sends any morphism to an isomorphism, in particular, it sends each triangle $X\xto{f} Y\xto{g} Z\to X[1]$ to an exact sequence $0\to F(X)\xto{F(f)} F(Y)\xto{F(g)} F(Z)$ with $F(f), F(g)$ being isomorphisms. 
The exactness forces $F(X)=F(Y)=F(Z)=0$.

The following shows that the perpendicular category of $\dfct\C$ coincides with $\lex\C$.

\begin{proposition}\label{prop:char_of_perp}
Let $\C$ be an extriangulated category with weak-kernels.
The following assertions hold.
\begin{itemize}
\item[(1)] Let $F\in\mod\C$. Then, $F$ sends deflations to monomorphisms if and only if $F\in(\dfct\C)^{\perp_0}$.
\item[(2)] We have an equality $(\dfct\C)^\perp =\lex\C$.
\end{itemize}
\end{proposition}
\begin{proof}
Let $Z\overset{g}{\longrightarrow} Y\overset{f}{\longrightarrow} X\overset{\delta}{\dashrightarrow}$ be an $\mathbb{E}$-triangle in $\C$ and consider an associated exact sequence
\begin{equation}\label{seq:defect1}
(-,Z)\to (-,Y)\to (-,X)\to \tilde\delta\to 0.
\end{equation}

(1) We assume that $F$ sends deflations to monomorphisms,
that is, we have a monomorphism $0\to F(X)\xto{F(f)} F(Y)$.
We shall show that $(\tilde\delta,F)=0$.
Taking a morphism $\alpha\in (\tilde\delta,F)$, we get the following commutative diagram:
$$
\xymatrix@C=18pt@R=25pt{
(-,Z)\ar[r]&(-,Y)\ar[r]^{(-,f)}\ar[rrd]_0&(-,X)\ar[rd]^\beta\ar@{->>}[r]&\tilde\delta\ar[d]^\alpha\ar[r]&0\\
&&&F&
}
$$
By the Yoneda Lemma, we have $F(f)(\beta)=0$.
The injectivity of $F(f)$ forces $\beta=0$. Hence we have $\alpha=0$, showing $F\in (\dfct\C)^{\perp_0}$.

Assume $F\in(\dfct\C)^{\perp_0}$.
To show the injectivity of $F(f)$, consider $x\in F(X)$ satisfying $F(f)(x)=0$.
Since the composed morphism $(-,Y)\xto{(-,f)} (-,X)\xto{x} F$ is zero,
$x$ factors through $\tilde\delta\in\dfct\C$.
Since $(F, \tilde\delta)=0$, we have $x=0$.

(2) Consider an object $F\in\lex\C$, i.e., we have an exact sequence $0\to F(X)\xto{F(f)}F(Y)\xto{F(g)}F(Z)$.
We shall show $\Ext^1(\tilde\delta,F)=0$.
By applying $(-,F)$ to an exact sequence $0\to G\to (-,X)\to \tilde\delta\to 0$ obtained from (\ref{seq:defect1}), we get an exact sequence
$$
0\to (\tilde\delta,F)\to F(X)\xto{\varphi} (G,F)\to \Ext^1(\tilde\delta,F)\to 0
$$
Thus, it is enough to show that the morphism $\varphi$ is surjective.
For any $h\in (G,F)$, we have the following commutative diagram
$$
\xymatrix@C=18pt@R=15pt{
(-,Z)\ar[rr]^{(-,g)}\ar[rrrdd]_0&&(-,Y)\ar[rr]^{(-,f)}\ar@{->>}[rd]_p&&(-,X)\ar@{.>}[ddl]^{h'}\\
&&&G\ar[d]_h\ar@{>->}[ru]&\\
&&&F&
}
$$
In fact, regarding $hp$ as an element in $F(Y)$, we get $F(g)(hp)=0$.
Since $F(X)\xto{F(f)}F(Y)\xto{F(g)}F(Z)$ is exact, we have an element $h'\in F(X)$ which satisfies $F(f)(h')=hp$ and corresponds to the dotted arrow in the above diagram.
Hence we conclude that $\varphi$ is surjective.
Thus $\Ext^1(\tilde\delta, F)=0$.
Hence we have $F\in (\dfct\C)^\perp$.

Fix $F\in (\dfct\C)^\perp$.
To show that $F(X)\xto{F(f)}F(Y)\xto{F(g)}F(Z)$ is exact,
we consider an element $y\in F(Y)$ with $F(g)(y)=0$.
We extract an exact sequence $(-,Z)\xto{(-,g)}(-,Y)\xto{p} G\to 0$ from (\ref{seq:defect1}).
The assumption implies that the composed morphism $y\circ (-,g)$ is zero.
Thus we get a morphism $x: G\to F$ such that $xp=y$, which are depicted as follows:
$$
\xymatrix@C=18pt@R=15pt{
(-,Z)\ar[rd]_0\ar[r]^{(-,g)}&(-,Y)\ar[d]^y\ar[r]^p&G\ar@{.>}[ld]^x\ar[r]&0\\
&F&&
}
$$
Applying the functor $(-,F)$ to the exact sequence $0\to G\xto{\iota} (-,X)\to \tilde\delta\to 0$, we have the following exact sequence
$$0\to (\tilde\delta,F)\to F(X)\xto{\psi}(G,F)\to\Ext^1(\tilde\delta,F).$$
Since $\tilde\delta\in\dfct\C$ and $F\in(\dfct\C)^{\perp}$, the morphism $\psi: F(X)\xto{\sim}(G,F)$ is an isomorphism.
We get $x':=\psi^{-1}(x)\in F(X)$.
Regarding $x'$ as a morphism $(-,X)\to F$, we have $y=x\circ p =(x'\circ \iota)\circ p =x'\circ (-,f)$.
This means $F(f)(x')=y$, showing the exactness at $F(Y)$.
Combining (1),  we conclude that $F$ is left exact.
\end{proof}

The following is our first result which is a generalization of Auslander's formula (\cite{Aus66}, \cite[IV. 4]{ARS}).

\begin{theorem}\label{thm:extri_Auslander's_formula}
Let $(\C,\mathbb{E}, \mathfrak{s})$ be an extriangulated category with weak-kernels.
Then, we have a Serre quotient
\begin{equation}\label{eq:Serre_quot_for_extri}
\xymatrix@C=1.2cm{
\dfct\C\ar[r]^-{{}}
&\mod\C\ar[r]^-{Q}
&\frac{\mod\C}{\dfct\C}.
}
\end{equation}
Moreover, if the quotient functor $Q$ has a right adjoint, then we have a localization sequence
\begin{equation}\label{eq:localiztion_seq_for_extri}
\xymatrix@C=1.2cm{
\dfct\C \ar[r]^-{{}}
&\mod\C\ar[r]^-{Q}\ar@/^1.2pc/[l]^-{}
&\lex\C\ar@/^1.2pc/[l]^{R}}
\end{equation}
where $R$ deotes the canonical inclusion.
\end{theorem}
\begin{proof}
Lemma \ref{lem:Serre_subcategory} and Proposition \ref{prop:char_of_Defect} show that the subcategory $\dfct\C$ is a Serre subcategory in $\mod\C$.
Thus we obtain the Serre quotient (\ref{eq:Serre_quot_for_extri}).
It remains to show that there is an equivalence $\frac{\mod\C}{\dfct\C}\simeq \lex\C$.
This follows from Lemma \ref{lem:property_of_localization_seq} and Proposition \ref{prop:char_of_perp}.
\end{proof}

Auslander's formula (\ref{eq:Auslander's formula}) shows that, if $\C$ is abelian, the quotient functor always has a right adjoint.
However, even if a given category $\C$ is exact, the quotient functor $Q$ does not necessarily have a right adjoint.
The author is grateful to Haruhisa
Enomoto for providing to him an example which shows the last fact.

\begin{example}\label{ex:Enomoto}
Let $R:=k[\![x, y]\!]$ be a ring of formal power series over a commutative field $k$.
We consider
\begin{itemize}
\item the category $\mod R$ of finitely generated modules;
\item the full subcategory $\proj R$ of all finitely generated projectives;
\item the Serre subcategory $\fl R$ of modules of finite length.
\end{itemize}
Then we have the quotient functor $Q:\mod R\to\frac{\mod R}{\fl R}$ which does not have a right adjoint.
In fact, if $Q$ has a right adjoint $R$, we have an equivalence $\frac{\mod R}{\fl R}\simeq (\fl R)^\perp$ by Lemma \ref{lem:property_of_localization_seq}.
Furthermore, it is basic that, for each $R$-module $X$, it belongs to $(\fl R)^\perp$ if and only if it satisfies $\Hom_R(k,X)=0=\Ext^1_R(k,X)$ if and only if $\mathsf{depth}X\geq 2$.
Since the Krull dimension of $R$ is two, we have $(\fl R)^\perp=\proj R$.
This contradicts to the fact that $\proj R$ is not abelian.

Thanks to \cite[Exa. 3.4, Thm. 3.7]{Eno18a}, there exists an exact structure on $\proj R$ which induces an equivalence $\fl R\simeq \dfct (\proj R)$.
Using an equivalence $\mod R\simeq \mod(\proj R)$, we have a Serre quotient $\mod(\proj R)\to\frac{\mod(\proj R)}{\dfct(\proj R)}$ which does not have a right adjoint.
\end{example}

Define the composed functor $E_\C:=Q\mathbb{Y}:\C\to\frac{\mod\C}{\dfct\C}$.
The following gives characterizations for $\C$ to be exact  or abelian via this functor.

\begin{theorem}\label{thm:exact_or_abelian}
Let $(\C, \mathbb{E}, \mathfrak{s})$ be an extriangulated category with weak-kernels.
Then the following hold.
\begin{itemize}
\item[(1)] The functor $E_\C$ is exact and fully faithful if and only if $\C$ is an exact category.
\item[(2)] The functor $E_\C$ is an exact equivalence if and only if $\C$ is an abelian category. If this is the case, we have an equivalence $\C\simeq\lex\C$.
\end{itemize}
\end{theorem}

In the rest of this subsection, we shall prove Theorem \ref{thm:exact_or_abelian}.

\begin{lemma}\label{lem:right_exactness_of_GQ}
The functor $E_\C$ is right exact.
\end{lemma}
\begin{proof}
Let $X\longrightarrow Y\longrightarrow Z\overset{\delta}{\dashrightarrow}$ be an $\mathbb{E}$-triangle in $\C$ which defines a defect $\tilde\delta$ by the exactness of a sequence $(-,X)\to (-,Y)\to (-,Z)\to \tilde\delta\to 0$ in $\mod\C$.
Since $\tilde\delta\in\S$ and the quotient functor $Q$ is exact, the assertion follows.
\end{proof}

\begin{lemma}\label{lem:exact_fully_faithful}
Suppose that $\C$ is an exact category with weak-kernels.
Then, the functor $E_\C:\C\to\frac{\mod\C}{\dfct\C}$ is exact and fully faithful.
\end{lemma}
\begin{proof}
We shall show the exactness of $E_\C$.
Note that a conflation $\delta: X\to Y\to Z$ in $\C$ gives rise to a projective resolution of the defect $\tilde\delta$.
By Lemma \ref{lem:right_exactness_of_GQ}, applying $Q$ to the projective resolution yields an exact sequence $0\to Q(-,X)\to Q(-,Y)\to Q(-,Z)\to 0$ in $\frac{\mod\C}{\dfct\C}$.

We shall show that $E_\C$ is fully faithful.
Since every representable functor in $\mod\C$ is left exact, it belongs to $\lex\C=(\dfct\C)^\perp$.
Due to Lemma \ref{lem:property_of_localization_seq}, the fully faithfulness of $E_\C$ follows.
\end{proof}

We recall the following result from \cite[p. 205]{Aus66}.

\begin{lemma}[Auslander's formula]\label{lem:Auslander's_defect}
Suppose that $\C$ is abelian.
Then, the Yoneda embedding $\mathbb{Y}:\C\hookrightarrow\mod\C$ admits an exact left adjoint $Q$.
Moreover, we have a localization sequence:
\begin{equation}\label{Auslander's_defect}
\xymatrix@C=1.2cm{
\dfct\C\ar[r]^-{{}}
&\mod\C\ar[r]^-{Q}\ar@/^1.2pc/[l]^-{}
&\C .\ar@/^1.2pc/[l]^{\mathbb{Y}}}
\end{equation}
\end{lemma}

Now we are ready to prove Theorem \ref{thm:exact_or_abelian}.

\begin{proof}[Proof of Theorem \ref{thm:exact_or_abelian}]
(1) The `if' part follows from Lemma \ref{lem:exact_fully_faithful}.
To show the `only if' part, we suppose that $E_\C$ is exact and fully faithful.
Thanks to Proposition \ref{prop:extri_to_exact}, we have only to show that, for each conflation $X\xto{f}Y\xto{g}Z$,
the morphism $f$ is a monomorphism and the morphism $g$ is an epimorphism.
Since $E_\C$ is exact, we get an exact sequence $0\to Q(-,X)\to Q(-,Y)\to Q(-,Z)\to 0$ in $\frac{\mod\C}{\dfct\C}$.
Since $E_\C$ is fully faithful, the assertion is obvious.

(2) If $\C$ is abelian, by comparing Auslander's defect formula (\ref{Auslander's_defect}) with the Serre quotient $\dfct\C\to\mod\C\to\frac{\mod\C}{\dfct\C}$,
we get an equality $\C=\frac{\mod\C}{\dfct\C} =\lex\C$ as subcategories in $\mod\C$.
The `only if' part is obvious because $\frac{\mod\C}{\dfct\C}$ is abelian.
\end{proof}

\subsection{The case of enough projectives}
We study the case that an extriangulated category $\C$ has enough projectives.

\begin{definition}
Let $(\C,\mathbb{E},\mathfrak{s})$ be an extriangulated category.
We say that \textit{$\C$ has enough projectives} if there exists a full subcategory $\P$ in $\C$ with $\mathbb{E}(\P,\C)=0$ and, for every $C\in\C$,
there exists a conflation $C'\to P\to C$ with $P\in\P$.
\end{definition}

In this case, we have nicer forms of the quotient functor $Q:\mod\C\to\frac{\mod\C}{\dfct\C}$ and the functor $E_\C:\C\to\frac{\mod\C}{\dfct\C}$.

\begin{proposition}\label{prop:gq_equivalent_to_functor_category}
Let $(\C, \mathbb{E}, \mathfrak{s})$ be an extriangulated category with weak-kernels which has enough projectives.
Let $\P$ be the subcategory of projectives in $\C$ and consider the restriction functor $\mathsf{res}_\P :\mod\C\to\mod\P$.
Then the following hold.
\begin{itemize}
\item[(1)] There exists an equivalence $Q':\frac{\mod\C}{\dfct\C}\simeq\mod\P$ with $\mathsf{res}_\P\cong Q'\circ Q$.
\item[(2)] The functor $E_\C$ is isomorphic to the restricted Yoneda functor $\mathbb{Y}_\P:\C\to\mod\P$.
\item[(3)] An equality $\dfct\C=\mod(\C/[\P])$ holds in $\mod\C$.
\end{itemize}
\end{proposition}
\begin{proof}
(1) Let $\tilde\delta$ be a defect corresponding to an $\mathbb{E}$-triangle $Z\longrightarrow Y\longrightarrow X\overset{\delta}{\dashrightarrow} $.
Since $\tilde\delta$ is a subobject of $\mathbb{E}(-, Z)$, it vanishes on $\P$.
Thus, the restriction functor $\mathsf{res}_\P:\mod\C\to\mod\P$ vanishes on $\dfct\C$.
Since $\mathsf{res}_\P$ is exact, by the universality, there uniquely exists an exact functor $Q':\frac{\mod\C}{\dfct\C}\to\mod\P$ such that $\mathsf{res}_\P\cong Q'\circ Q$.
We shall show that $Q'$ is an equivalence.

It is obvious that $Q'$ is full and dense.
To show the faithfulness of $Q'$, we consider a morphism $\alpha: F\to G$ in $\mod\C$ such that $\mathsf{res}_\P(\alpha)=0$.
Note that there exists an epimorphism $(-,C)\to F\to 0$.
Since $\C$ has enough projectives, we have an $\mathbb{E}$-triangle $C'\longrightarrow P\longrightarrow C\overset{\delta}{\dashrightarrow}$ with $P\in\P$.
The following commutative diagram shows that the functor $\Im\alpha$ belongs to $\dfct\C$.
$$
\xymatrix@C=18pt@R=15pt{
(-,C')\ar[r]&(-,P)\ar[rdd]_0\ar[r]&(-,C)\ar@{->>}[d]\ar@{->>}[r]&\tilde\delta\ar[r]\ar@{->>}[ldd]&0\\
&&F\ar@{->>}[d]_\alpha&&\\
&&\Im\alpha&&
}
$$
Thus we get $\Im\alpha\in\dfct\C$,
which shows that the morphisms $\alpha$ factors through an object in $\dfct\C$.
This shows the faithfulness of $Q'$.

(2) This directly follows from (1). In fact, we have $Q'E_\C=Q'Q\mathbb{Y}\cong \mathsf{res}_\P\mathbb{Y}=\mathbb{Y}_\P$.

(3) Since $\P$ is contravariantly finite in $\C$, by Example \ref{ex:colocalization_of_functor_category}(1), we have a Serre quotient $\mod(\C/[\P])\to\mod\C\to\mod\P$.
By comparing it with a Serre quotient $\dfct\C\to\mod\C\to\frac{\mod\C}{\dfct\C}$, we have a desired equality.
\end{proof}

We end this section by showing that, in the case that $\C$ is an exact category having enough projectives,
the quotient functor $Q:\mod\C\to\frac{\mod\C}{\dfct\C}\simeq\mod\P$ always admits a right adjoint.

\begin{proposition}\label{prop:fpGQ}
Let $(\C,\mathbb{E})$ be an exact category with weak-kernels which has enough projectives.
Then,
the restriction functor $\mathsf{res}_\P:\mod\C\to\mod\P$ admits a right adjoint $R$. Moreover, it induces a recollement
\begin{equation}
\xymatrix@C=1.2cm{
\dfct\C\ar[r]^-{{}}
&\mod\C\ar[r]^-{\mathsf{res}_\P}\ar@/^1.2pc/[l]^-{}\ar@/^-1.2pc/[l]^-{}
&\mod\P\ar@/^1.2pc/[l]^{R}\ar@/^-1.2pc/[l]_{L}
}
\end{equation}
\end{proposition}
\begin{proof}
It is basic that $\mathsf{res}_\P$ is full and dense.
Clearly, the subcategory $\C$ is contravariantly finite in  $\mod\C$.
By Proposition \ref{prop:gq_equivalent_to_functor_category}, the composition $\mathbb{Y}_\P:\C\hookrightarrow \mod\P$ is fully faithful.
It is straightforward that the subcategory $\C$ is a contravariantly finite in $\mod\P$.
Therefore, by Example \ref{ex:colocalization_of_functor_category}, we have a colocalization sequence $\mod(\mod\P/[\C])\to\mod(\mod\P)\xto{\mathsf{res}_\C}\mod\C$, that is, $\mathsf{res}_\C$ has a left adjoint $L$.
Let $\mathbb{Y}:\mod\P\hookrightarrow \mod(\mod\P)$ be the Yoneda functor.
It is easily checked that the functor $R:=\mathsf{res}_\C\circ\mathbb{Y}:\mod\P\to\mod\C$ is a right adjoint of $\mathsf{res}_\P$.
By Proposition \ref{prop:right_adjoint1} and its dual, we have a desired recollement.
\end{proof}

We should mention that there are related results for exact categories with enough projectives \cite[Prop. A]{Eno17} and Frobenius exact categories \cite[Thm. 4.2]{Che10}.

%%%%%%%%%%%%%%%%%%%%%%%%%%%%%%%%%%%%%%%%%%%%%%%%%%%%%%%%
\section{Direct colimits of Auslander's defects}\label{sec:GQ}
%%%%%%%%%%%%%%%%%%%%%%%%%%%%%%%%%%%%%%%%%%%%%%%%%%%%%%%%
Let $\C$ be a skeletally small extriangulated category with weak-kernels.
We consider an expansion of $\eff\C$ by taking direct colimits.
To begin, we recall basic propeties of direct colimits in $\Mod\C$.
A poset $(I,\leq)$ is called a \textit{directed set} if,
for any $i,j\in I$ there exists $k\in I$ with $i\leq k, j\leq k$.
Regarding a directed set $I$ as a category,
for a covariant functor $F:I\to\Mod\C$,
the associated colimit ${\underrightarrow{\lim}}_I F_i$ is called a \textit{direct colimit} of $\{F_i\}_{i\in I}$.
For any colimit ${\underrightarrow{\lim}}_I F_i$, there exists an exact sequence
$$\coprod_{k\in K}F_k\to\coprod_{j\in J}F_j\to{\underrightarrow{\lim}}_I F_i\to 0$$
for some sets $K$ and $J$.
For any skeletally small additive category $\C$, it is well-known that 
any object in $\Mod\C$ can be obtained as a direct colimit of objects in $\mod\C$.
We denote by $\Dfct\C$ the full subcategory in $\Mod\C$ consisting of direct colimits of objects in $\dfct\C$.
Our first aim of this section is to expand Theorem \ref{thm:extri_Auslander's_formula} by taking direct colimits as follows:

\begin{theorem}\label{thm:GQ_for_extri}
Let $(\C,\mathbb{E}, \mathfrak{s})$ be a skeletally small extriangulated category with weak-kernels.
Then, the Serre quotient
\begin{equation*}\label{eq:Serre_quot_for_extri2}
\xymatrix@C=1.2cm{
\dfct\C\ar[r]^-{{}}
&\mod\C\ar[r]^-{Q}
&\frac{\mod\C}{\eff\C}.
}
\end{equation*}
induces the following localization sequence
\begin{equation}\label{eq:localization_seq_for_extri2}
\xymatrix@C=1.2cm{
\Dfct\C \ar[r]^-{{}}
&\Mod\C\ar[r]^-{}\ar@/^1.2pc/[l]^-{}
&\Lex\C\ar@/^1.2pc/[l]^{R}}
\end{equation}
where $R$ denotes the canonical inclusion.
\end{theorem}
The first Serre quotient is given in Theorem \ref{thm:extri_Auslander's_formula}.
We have only to show that the localization sequence (\ref{eq:localization_seq_for_extri2}) exists.
To show $\Dfct\C$ is a Serre subcategory in $\Mod\C$.
We consider the effaceable functors in $\Mod\C$, that is,
we say a functor $F\in\Mod\C$ to be \textit{effaceable} if it satisfies the condition (\ref{condition_eff}) in Definition \ref{def:effaceable} and denote by $\Eff\C$ the full subcategory of effaceable functors.
Clearly $\eff\C=\mod\C\cap\Eff\C$ holds.
In the following, we shall show that $\Eff\C$ coincides with $\Dfct\C$.

\begin{lemma}\label{lem:Eff_is_localizing}
Let $\C$ be a skeletally small extriangulated category with weak-kernels.
Then the subcategory $\Eff\C$ is localizing in $\Mod\C$.
\end{lemma}
\begin{proof}
By a similar argument given in Lemma \ref{lem:Serre_subcategory}, we can easily check that $\Eff\C$ is a Serre subcategory in $\Mod\C$.

Thanks to Lemma \ref{prop:right_adjoint3}, it suffices to show that $\Eff\C$ is closed under coproducts.
Let $\{S_i\}_{i\in I}$ be a set of objects in $\Eff\C$ and put $F:=\coprod_{i\in I}S_i$.
Consider $X\in\C$ and $x:=\{x_i\}_{i\in I}\in F(X)$.
There exists a finite subset $J\subseteq I$ such that $x_i=0$ for $i\in I\setminus J$.
For each $j\in J$, since $S_j\in\widetilde{\S}$, we get a deflation $Y_j\xto{\alpha_j} X$
such that $S_j(\alpha_j)(x_j)=0$.
Thus we have a deflation $\coprod_{j\in J}Y_j\xto{\alpha} \coprod_{j\in J}X$.
Let $\pi:\coprod_{j\in J}X\to X$ be a natural summation morphism.
Since $\pi$ is a splitting epimorphism, we have a deflation $\pi\alpha:\coprod_{j\in J}Y_j\to X$.
By a standard argument, we can verify that $F(\pi\alpha)(x)=0$. This shows $F \in\Eff\C$.
\end{proof}

\begin{lemma}\label{lem:colimit_of_eff_is_Eff}
An equality $\Dfct\C=\Eff\C$ holds in $\Mod\C$.
\end{lemma}
\begin{proof}
Let $S$ be an object in $\Dfct\C$.
Then there exists an exact sequence
\begin{equation}\label{seq:colimit}
\coprod_{i\in I}S_i\to\coprod_{j\in J}S_j\to S\to 0
\end{equation}
in $\Mod\C$ with $S_i, S_j\in\dfct\C$ for any $i\in I, j\in J$.
By Lemma \ref{lem:Eff_is_localizing}, we get $S\in\Eff\C$.

Let $S$ be an object in $\Eff\C$ with a set of morphisms $\{f_i:(-,X_i)\to S\}_{i\in I}$ forming an epimorphism $\coprod_{i\in I}(-,X_i)\to S\to 0$.
Since $S\in\Eff\C$, for each $i\in I$, there exists a deflation $\alpha_i: Y_i\to X_i$ which satisifies that the composition $(-,Y_i)\xto{(-,\alpha_i)} (-,X_i)\xto{f_i}S$ is zero.
We consider a set of exact sequences $\{(-,Y_i)\xto{(-,\alpha_i)}(-,X_i)\to S_i\to 0\}_{i\in I}$
and the coproduct of them:
$$\coprod_{i\in I}(-,Y_i)\xto{}
\coprod_{i\in I}(-,X_i)\to \coprod_{i\in I}S_i\to 0.$$
Since each $S_i$ belongs to $\dfct\C$, we have $\coprod S_i\in\Dfct\C$.
Since $S$ is a factor object of $\coprod S_i$, we conclude that $S$ also belongs to $\Dfct\C$.
\end{proof}

Now we are in position to prove Proposition \ref{thm:GQ_for_extri}.

\begin{proof}[Proof of Proposition \ref{thm:GQ_for_extri}]
By Lemmas \ref{lem:Eff_is_localizing} and \ref{lem:colimit_of_eff_is_Eff}, we have a localizaing subcategory $\Dfct\C$ which gives rise to a localization sequence
\begin{equation*}
\xymatrix@C=1.2cm{
\Dfct\C \ar[r]^-{{}}
&\Mod\C\ar[r]^-{Q}\ar@/^1.2pc/[l]^-{}
&\frac{\Mod\C}{\Dfct\C}.\ar@/^1.2pc/[l]^{R}}
\end{equation*}
It remains to show an equality $(\Dfct\C)^\perp=\Lex\C$ in $\Mod\C$.
Let $F\in\Mod\C$.
By a similar argument in Proposition \ref{prop:char_of_perp}(2),
we can easily check $\Lex\C\subseteq (\dfct\C)^\perp$ and $(\Dfct\C)^\perp\subseteq\Lex\C$.
It remains to check a containment $(\dfct\C)^\perp\subseteq(\Dfct\C)^\perp$.
Applying $(-,F)$ to the sequence (\ref{seq:colimit}), we have exact sequences
$$0\to(S,F)\to\prod_{j\in J}(S_j,F)\ \ \ \textnormal{and}\ \ \ 
0\to\Ext^1_\C(S,F)\to\prod_{j\in J}\Ext^1_\C(S_j,F).$$
Since $(S_j,F)=\Ext^1_\C(S_j,F)=0$ for any $j\in J$, we conclude $F\in(\Dfct\C)^\perp$.
Thanks to Lemma \ref{lem:property_of_localization_seq}(3), we have an equivalence $\frac{\Mod\C}{\Dfct\C}\simeq \Lex\C$.
This finishes the proof.
\end{proof}

We end this section by mentioning that,
in the case that $\C$ is exact,
the localization sequence (\ref{eq:localization_seq_for_extri2}) provides the Gabriel-Quillen embedding functor.
The theorem is stated as below.
A proof is provided in \cite[Appendix B]{Kel90} (see also \cite[Thm. A.1]{Buh10}).

\begin{proposition}
Let $\C$ be a skeletally small exact category.
Then the following hold.
\begin{itemize}
\item[(1)] There exists a localization sequence
\begin{equation*}
\xymatrix@C=1.2cm{
\Eff\C \ar[r]^-{{}}
&\Mod\C\ar[r]^-{Q}\ar@/^1.2pc/[l]^-{}
&\Lex\C\ar@/^1.2pc/[l]^{R}}
\end{equation*}
where $R$ denotes a canonical inclusion.
\item[(2)] The composition functor $Q\circ\mathbb{Y}:\C\to\Lex\C$ is a fully faithful exact functor which reflects exactness.
This functor is known as the \textnormal{Gabriel-Quillen embedding functor}.
\end{itemize}
\end{proposition}

Since $\Eff\C=\Dfct\C$,
if $\C$ is an exact category with weak-kernels, then the above localization sequence coincides with the one (\ref{eq:localization_seq_for_extri2}) in Theorem \ref{thm:GQ_for_extri}.

%%%%%%%%%%%%%%%%%%%%%%%%%%%%%%%%%%%%%%%%%%%%%%%%%%%%%%%%
\section{General heart construction via Left exact functors}\label{sec:heart}
%%%%%%%%%%%%%%%%%%%%%%%%%%%%%%%%%%%%%%%%%%%%%%%%%%%%%%%%
Throughout this section, we fix a triangulated category $\T$ with a translation $[1]$.
A pair $(\U,\V)$ of full subcategories in $\T$ is called a \textit{cotorsion pair} if it is closed under direct summands and satisfies the following conditions:
\begin{itemize}
\item $\T(U,V[1])=0$ for any $U\in\U$ and $V\in\V$;
\item For any object $X$ of $\T$, there exists a triangle $U\to X\to V[1]\to U$ with $U\in\U$ and $V\in\V$.
\end{itemize}
Since $\U$ is extension-closed and contravariantly finite in $\T$, it gives rise to an extriangulated category with weak-kernels by setting $\mathbb{E}(+,-):=\U(+,-[1])$.
First we shall show that the quotient functor $Q:\mod\U\to\frac{\mod\U}{\dfct\U}$ has a right adjoint, and hence we have an equivalence $\frac{\mod\U}{\dfct\U}\simeq \lex\U$.
Let us start with the following easy lemma.

\begin{lemma}\label{lem:char_of_representable}The following hold.
\begin{itemize}
\item[(1)] For any $V\in\V$, the functor $(-, V[2])|_{\U}$ belongs to $(\dfct\U)^{\perp}$.
\item[(2)] For any $U\in\U$, the functor $(-, U[1])|_{\U}$ belongs to $\dfct\U$.
\end{itemize}
\end{lemma}
\begin{proof}
(1) Since $(\U,V[1])=0$, the functor $(-,V[2])|_{\U}$ is left exact.

(2) We shall check that the functor $(-, U[1])|_{\U}$ satisfies the condition (\ref{condition_eff}).
Consider $X\in\U$ and $x\in (X, U[1])$.
Then we get a conflation $U\to U'\xto{\alpha} X$ which is a part of a triangle $U\to U'\xto{\alpha} X\xto{x}U[1]$ in $\T$.
Obviously $(X, U[1])\xto{-\circ\alpha}(U', U[1])$ sends $x$ to $\alpha\circ x=0$, which shows $(-,U[1])|_{\U} \in\dfct\U$.
\end{proof}

The next proposition shows the existence of a right adjoint of the quotient functor $Q$.

\begin{proposition}\label{prop:existence_of_right_adjoint}
The quotient functor $Q:\mod\U\to\frac{\mod\U}{\dfct\U}$ has a right adjoint $R$.
Moreover, $R$ induces an equivalence $R:\frac{\mod\U}{\dfct\U}\xto{\sim}\lex\U$.
\end{proposition}
\begin{proof}
Let $F$ be an object in $\mod\U$ with a projective presentation $(-, U_1)\xto{f\circ -} (-, U_0)\to F\to 0$.
Thanks to Proposition \ref{prop:right_adjoint2}, it is enough to show that there exists an exact sequence $S\to F\to G$ with $S\in\dfct\U$ and $G\in (\dfct\U)^{\perp}$.
The morphism $f$ induces a triangle $K\to U_1\xto{f} U_0\to K[1]$ in $\T$.
For the object $K$, there exists a triangle $U_2\to K\to V_2[1]\to U_2[1]$ with $U_2\in\U$ and $V_2\in\V$ induced by the cotorsion pair $(\U,\V)$.
We complete the composed morphism $U_2\to K\to U_1$ to a triangle
\begin{equation}\label{tri:1}
U_2\to U_1\to C\to U_2[1],
\end{equation}
and then we have the commutative diagram below
\begin{equation}\label{diag:existence_of_right_adjoint1}
\xymatrix@C=18pt@R=15pt{
&&V_2[1]\ar[d]\ar@{=}[r]&V_2[1]\ar[d]\\
U_2\ar[r]\ar[d]&U_1\ar@{=}[d]\ar[r]&C\ar[r]\ar[d]^g&U_2[1]\ar[d]\\
K\ar[r]&U_1\ar[r]&U_0\ar[d]\ar[r]&K[1]\ar[d]\\
&&V_2[2]\ar@{=}[r]&V_2[2]
}
\end{equation}
where all rows and columns are triangles in $\T$.
The third column induces an exact sequence
\begin{equation}\label{seq:existence_of_right_adjoint1}
0\to (-,C)|_\U\xto{g\circ -} (-,U_0)\to (-,V_2[2])|_\U
\end{equation}
Thus we have a monomorphism $\varphi:=(g\circ -):(-,C)|_\U\rightarrowtail (-,U_0)$.
The second and third rows in (\ref{diag:existence_of_right_adjoint1}) induce the following commutative diagram
$$
\xymatrix@C=18pt@R=18pt{
(-,U_2)\ar[r]\ar[d]&(-,U_1)\ar[r]\ar@{=}[d]&(-,C)|_\U\ar@{}[rd]|{(*)}\ar[r]\ar@{>->}[d]_{\varphi}&S\ar[r]\ar@{>->}[d]^\phi&0\\
(-,K)|_\U\ar[r]&(-,U_1)\ar[r]^{f\circ -}&(-,U_0)\ar[r]&F\ar[r]&0
}
$$
with exact rows.
Since the square $(*)$ is a pullback, we have an isomorphism $\Cok\varphi\cong\Cok\phi$.
By the sequence (\ref{seq:existence_of_right_adjoint1}), we have an exact sequence
$$S\xto{\phi}F\to (-,V_2[2])|_\U$$
in $\mod\U$.
Lemma \ref{lem:char_of_representable} shows $(-,V_2[2])|_\U\in\ (\dfct\U)^\perp$ and $(-,U_2[1])|_\U\in\dfct\U$.
Since $S$ is a subobject of $(-,U_2[1])|_\U$, we get $S\in\dfct\U$.
This finishes the proof.
\end{proof}

By combining Proposition \ref{prop:existence_of_right_adjoint} and Theorem \ref{thm:GQ_for_extri},
we have the following localization sequence.

\begin{corollary}\label{cor:localization_seq_for_cotorsion_class}
For a cotorsion pair $(\U, \V)$ in a triangulated category $\T$,
there exists a localization sequence
\begin{equation*}
\xymatrix@C=1.2cm{
\dfct\U\ar[r]^-{{}}
&\mod\U\ar[r]^-{Q}\ar@/^1.2pc/[l]^-{}
&\lex\U\ar@/^1.2pc/[l]^{R}}
\end{equation*}
where $R$ denotes the canonical inclusion.
\end{corollary}

Finally we study a connection between $\lex\U$ and the heart of the cotorsion pair $(\U,\V)$.
Let us introduce the following notion:
For two classes $\U$ and $\V$ of objects in
$\T$, we denote by $\U *\V$ the class of objects $X$ occurring in a triangle $U\to X\to V\to U[1]$
with $U\in\U$ and $V\in\V$.

\begin{definition}
Let $(\U,\V)$ be a cotorsion pair in a triangulated category $\T$.
We define the following associated categories:
\begin{itemize}
\item Put $\W:=\U\cap \V$;
\item For a sequence $\W\subseteq\S\subseteq\T$ of subcategories, we put $\underline{\S}:=\S/[\W]$ and denote by $\pi:\S\to\underline{\S}$ the canonical ideal quotient functor;
\item We put $\T^+:=\W*\V[1], \T^-:=\U[-1]*\W$ and $\H:=\T^+\cap\T^-$.
\end{itemize}
We call the category $\underline{\H}$ the \textit{heart of $(\U,\V)$}.
\end{definition}

As mentioned in Introduction, the heart $\underline{\H}$ is abelian and there exists a \textit{cohomological} functor $\mathbb{H}:\T\to \underline{\H}$, namely, $\mathbb{H}$ sends any triangle $X\to Y\to Z\to X[1]$ in $\T$ to an exact sequence $\mathbb{H}X\to \mathbb{H}Y\to \mathbb{H}Z\to \mathbb{H}X[1]$ in $\underline{\H}$. 
We recall a construction of the cohomological functor $\mathbb{H}:\T\to\underline{\H}$. 
By definition we have the following containments
$$
\xymatrix@C=24pt@R=6pt{
&\underline{\T}^+\ar@{^{(}->}[rd]|{i_+}&\\
\underline{\H}\ar@{^{(}->}[ru]\ar@{^{(}->}[rd]\ar@{}[rr]|\circlearrowleft&&\underline{\T}\\
&\underline{\T}^-\ar@{^{(}->}[ru]|{i_-}&
}$$
where the arrows denote the canonical inclusions.
It is a crucial property of the above inclusions that $i_+$ and $i_-$ admit a left adjoint and a right adjoint, respectively.
First, the following statement says that there exists a nice left $(\T^+)$-approximation (resp. right $(\T^-)$-approximation) for each $X\in\T$.

\begin{proposition}\label{prop:reflection}\cite[Prop. 4.3]{Nak11}
Let $X$ be an object in $\T$.
\begin{itemize}
\item[(a)] There exists a diagram of the form
\begin{equation}\label{eq:reflection}
\xymatrix@C=18pt@R=10pt{
U'[-1]\ar[rr]\ar[rd]&&X\ar[rr]^{\alpha^+}&&X^+\ar[rr]&&U'\\
&U_X\ar[ru]&&&&&
}
\end{equation}
with the first row forming a triangle, $U',U_X\in\U$ and $X^+\in\T^+$.
We call this triangle and the morphism $\alpha^+$ a \textnormal{reflection triangle for $X$} and a \textnormal{reflection morphism for $X$}, respectively.
\item[(b)] There exists a diagram of the form
\begin{equation}\label{eq:coreflection}
\xymatrix@C=18pt@R=10pt{
V'\ar[rr]&&X^-\ar[rr]^{\alpha^-}&&X\ar[rr]\ar[rd]&&V'[1]\\
&&&&&V_X\ar[ru]&
}
\end{equation}
with the first row forming a triangle, $V',V_X\in\V$ and $X^-\in\T^-$.
We call this triangle and the morphism $\alpha^-$ a \textnormal{coreflection triangle for $X$} and a \textnormal{coreflection morphism for $X$}, respectively.
\end{itemize}
In particular, the subcategory $\T^+$ (resp. $\T^-$) is covariantly finite (resp. contravariantly finite) in $\T$.
\end{proposition}

By \cite[Lem. 4.6]{Nak11}, if $X\in\T^-$, then the object $X^+$ in (\ref{eq:reflection}) belongs to $\H$.
Dually, $X\in\T^+$ forces $X^-\in\H$.
Furthermore, by taking the stable categories with respect to $\W$, the assignments $X\mapsto X^+$ and $X\mapsto X^-$ behave very nicely as follows.

\begin{proposition}\label{prop:reflection_gives_adjoint}\cite[Cor. 4.4, Thm. 5.7]{Nak11}
\begin{itemize}
\item[(a)] The assignment $X\mapsto X^+$ gives rise to a functor $\tau^+:\underline{\T}\to\underline{\T}^+$.
Moreover, the functor $\tau^+$ is a left adjoint of the inclusion $i_+:\underline{\T}^+\hookrightarrow\underline{\T}$.
\item[(b)] The assignment $X\mapsto X^-$ gives rise to a functor $\tau^-:\underline{\T}\to\underline{\T}^-$.
Moreover, the functor $\tau^-$ is a right adjoint of the inclusion $i_-:\underline{\T}^-\hookrightarrow\underline{\T}$.
\end{itemize}
Furthermore, we have an isomorphism $\tau^-\circ\tau^+\cong \tau^+\circ\tau^-$ and the functor $\mathbb{H}:=\tau^-\circ\tau^+$ is cohomological.
\end{proposition}

By Corollary \ref{cor:localization_seq_for_cotorsion_class}, we have a localization sequence of $\mod\U[-1]$ relative to $\dfct\U[-1]$ and the functor $E_{\U[-1]}:\U[-1]\to\lex\U[-1]$.
We consider the following diagram:
\begin{equation}\label{diag:lex_vs_heart}
\xymatrix{
\H\ar[d]_\pi\ar@{^{(}-_{>}}[r]&\T\ar[r]^{\mathbb{Y}_{\U[-1]}\ \ \ \ \ }&\mod\U[-1]\ar[r]^{Q}&\lex\U[-1]\\
\underline{\H}\ar@{.>}[rrru]_{\Psi}&&
}
\end{equation}
There uniquely exists a dotted arrow $\Psi$ which makes the diagram commutative up to isomorphism.
In fact, the functor $\H\to\lex\U[-1]$ composed with all horizontal arrows sends any $W\in\W$ to zero, because of $(\U[-1], W)=0$.
Thus the functor $\Psi: \underline{\H}\to\lex\U[-1]$ sends $\pi (H)$ to $(-,H)|_{\U[-1]}$ for any $H\in\H$.

The aim of this section is to prove the following result.

\begin{theorem}\label{thm:heart_is_lex}
Let $(\U,\V)$ be a cotorsion pair  in a triangulated category $\T$.
Then the following hold.
\begin{itemize}
\item[(1)] There exists a natural equivalence $\Psi: \underline{\H}\xto{\sim}\lex\U[-1]$ which sends $\pi (H)$ to $(-,H)|_{\U[-1]}$ for any $H\in\H$.
\item[(2)] The cohomological functor $\mathbb{H}$ is isomorphic to the composed functor $\T\to\mod\U[-1]\xto{Q}\lex\U[-1]\xto{\Psi^{-1}}\underline{\H}$.
\end{itemize}
\end{theorem}

To prove Theorem \ref{thm:heart_is_lex}, we use the following lemma.

\begin{lemma}\label{lem:reflection}
Let $X$ be an object in $\T$.
Then there exist morphisms $X\xto{\alpha^+}X^+\xleftarrow{\alpha^-}X^\pm$ with $X^+\in\T^+$ and $X^\pm\in\H$ so that
they induce isomorphisms
$$Q(-,X)|_{\U[-1]}\xto{\sim}Q(-,X^+)|_{\U[-1]}\xleftarrow{\sim}Q(-,X^\pm)|_{\U[-1]}$$
in $\lex\U[-1]$.
Moreover, we have an isomorphism $\mathbb{H}(X)\cong \pi (X^\pm)$.
\end{lemma}
\begin{proof}
For a given $X\in\T$, we consider a reflection triangle (\ref{eq:reflection}) for $X$.
Since Lemma \ref{lem:char_of_representable}(3) tells $(-,U')|_{\U[-1]}, (-.U_X)|_{\U[-1]}\in\dfct\U[-1]$, the induced morphism $Q(-,\alpha^+)|_{\U[-1]}$ is an isomorphism in $\lex\U[-1]$.
Similarly, considering a coreflection triangle for $X^+$, we have a desired morphism.
We skip the details.
\end{proof}

\begin{proof}[Proof of Theorem \ref{thm:heart_is_lex}]
(1) For readability purposes we divide the proof in some steps.

\begin{claim}\label{claim:dense}
The functor $\Psi$ is dense.
\end{claim}
\begin{proof}
Let $F$ be an object in $\mod\U[-1]$ with projective presentation
$(-,U'[-1])\xto{(-,f)}(-,U[-1])\to F\to 0$.
Complete triangles $\eta:U'[-1]\xto{f}U[-1]\to X\to U$ in $\T$.
Since $(-,U)|_{\U[-1]}\in\dfct\U[-1]$ by Lemma \ref{lem:char_of_representable}(2), we have the following exact sequence
$$
Q(-,U'[-1])\xto{Q(-,f)}Q(-,U[-1])\to Q(-,X)|_{\U[-1]}\to 0
$$
in $\lex\U[-1]$.
Obviously we get $QF\cong Q(-,X)|_{\U[-1]}$.

By Lemma \ref{lem:reflection}, there exist morphisms $X\xto{\alpha^+}X^+\xleftarrow{\alpha^-}X^\pm$ with $X^\pm\in\H$.
Applying $\mathbb{Y}_{\U[-1]}\circ Q$ to this diagram,
we have isomorphims
$Q(-,X)|_{\U[-1]}\xto{\sim}Q(-,X^+)|_{\U[-1]}\xleftarrow{\sim}Q(-,X^\pm)|_{\U[-1]}$.
Since $X^\pm\in\H$, this shows that $\Psi$ is dense.
\end{proof}

\begin{claim}\label{claim:left_ex}
The functor $(-,H)|_{\U[-1]}$ is left exact for any $H\in\H$.
\end{claim}
\begin{proof}
Since $(-,H)|_{\U[-1]}$ is half exact, thanks to Proposition \ref{prop:char_of_perp}(1),
it suffices to show that $(-,H)|_{\U[-1]}$ sends a deflation to a monomorphism.
To this end, let $U'\to U\to U''$ be a conflation in $\U$, equivalently, there exists a triangle $U'[-1]\to U[-1]\xto{} U''[-1]\xto{a} U'$.
Applying the functor $(-,H)$ to the above triangle, we get an exact sequence
$$(U',H)\xto{-\circ a} (U''[-1],H)\xto{}(U[-1],H)\to (U'[-1],H)$$
To show the morphism $(U',H)\xto{-\circ a} (U''[-1],H)$ is zero, we take an element $b\in (U', H)$.
Since $H\in \W*\V[1]$, we have the following commutative diagram
$$
\xymatrix@C=18pt@R=18pt{
U'[-1]\ar[r]&U[-1]\ar[r]&U''[-1]\ar[r]^a&U'\ar@{.>}[dl]\ar[d]^b\ar[dr]^0&&\\
&&W\ar[r]&H\ar[r]&V[1]\ar[r]&W[1]
}
$$
with $W\in\W$ and $V\in\V$.
Thus the morphism $b\circ a:U''[-1]\to H_1$ factors through $W$, which shows $b\circ a =0$.
We have then concluded $(-,H)|_{\U[-1]}\in\lex\U[-1]$.
\end{proof}

\begin{claim}\label{claim:full}
The functor $\Psi$ is full.
\end{claim}
\begin{proof}
Let $H_1$ and $H_2$ be objects in $\H$.
By Claim \ref{claim:left_ex}, $(-, H_i)|_{\U[-1]}$ belongs to $\Lex\U[-1]$ for $i=1,2$.
Thus, it suffices to show that, for any morphism $\phi:(-,H_1)|_{\U[-1]}\to(-,H_1)|_{\U[-1]}$, there exists a morphism $c:H_1\to H_2$ such that $\Psi(c)\cong \phi$.
Since $H_i\in \U[-1]*\W$, there exist triangles $\eta_i:W_i[-1]\to U_i[-1]\to H_i\to W_i$ with $U_i\in\U$ and $W_i\in\W$ for each $i=1,2$.
These triangles induce projective presentations $(-,W_i[-1])\to (-,U_i[-1])\to (-,H_i)|_{\U[-1]}\to 0$ in $\mod\U[-1]$.
Then the morphism $\phi$ induces a morphism $(a,b,c):\eta_1\to\eta_2$ between triangles.
Hence we have an isomorphism $\Psi(c)\cong \phi$.
\end{proof}

\begin{claim}\label{claim:faithful}
The functor $\Psi$ is faithful.
\end{claim}
\begin{proof}
Let $h:H_0\to H_1$ be a morphism in $\H$ such that the induced morphism
$$(-,H_0)|_{\U[-1]}\xto{h\circ -}(-,H_1)|_{\U[-1]}$$
factors through an object in $\dfct\U[-1]$.
Since $(-,H_0)|_{\U[-1]}\in\lex\U[-1]$ by Claim \ref{claim:left_ex},
this morphism $h\circ -$ should be zero.
Since $H_0\in\U[-1]*\W$, we have the following commutative diagram
$$
\xymatrix@C=18pt@R=18pt{
U_0[-1]\ar[r]\ar[rd]_0&H_0\ar[r]\ar[d]^{h}&W_0\ar@{.>}[ld]\ar[r]&U_0\\
&H_1&&
}
$$
with $U_0\in\U$ and $W_0\in\W$,
showing that $h$ factors through $W_0$.
\end{proof}
Combining the above claims, we conclude that $\Psi$ is an equivalence.

(2) Consider a morphism $c:X_1\to X_2$ in $\T$.
By Proposition \ref{prop:reflection}, we have the following commutative diagram
\begin{equation}\label{diag:reflection_coreflection}
\xymatrix@C=18pt@R=15pt{
X_1\ar[d]_c\ar[r]^{\alpha_1^+}&X_1^+\ar[d]^{c^+}&X_1^\pm\ar[l]_{\alpha_1^-}\ar[d]^{c^\pm}\\
X_2\ar[r]^{\alpha_2^+}&X_2^+&X_2^\pm\ar[l]_{\alpha_2^-}
}
\end{equation}
with $X^+_i\in\T^+$ and $X^\pm_i\in\H$ for $i=1,2$.
Denote by $\Phi$ the composed functor $\T\to\mod\U[-1]\xto{Q}\lex\U[-1]\xto{\Psi^{-1}}\underline{\H}$.
Applying $\Phi$ to (\ref{diag:reflection_coreflection}), by Lemma \ref{lem:reflection},
we have an isomorphism $\Phi(c)\cong \Phi(c^\pm)$.
By the definition of the cohomological functor, we have $\pi(c^\pm)=\tau^+\tau^-(c)=\mathbb{H}(c)$.
By the commutativity of (\ref{diag:lex_vs_heart}), we have an isomorphism $\Phi(c^\pm)\cong\pi(c^\pm)$.
Hence we have an isomorphism $\Phi\cong\mathbb{H}$.
\end{proof}

By combining Proposition \ref{prop:gq_equivalent_to_functor_category} and Theorem \ref{thm:heart_is_lex}, we recover the following result.

\begin{corollary}\cite[Prop. 4.15]{LN19}\cite[Thm. 2.10]{Liu15}
Let $(\U,\V)$ be a cotorsion pair  in a triangulated category $\T$ and $\P$ the full subcategory of projectives in the extrianguated category $\U$.
If $\U$ has enough projectives, then we have an equivalence $\underline{\H}\xto{\sim}\mod\P$.
\end{corollary}

\subsection*{Acknowledgements}
The author wishes to thank Professor Hiroyuki Nakaoka for his useful comments. 
He also wishes to thank Haruhisa Enomoto for having a helpful discussion about the exact category.

%%%%%%%%%%%%%%%%%%%%%%%%%%%%%%%%%%%%%%%%%%%%%%%%%%%%%%%%

%%%%%%%%%%%%%%%%%%%%%%%%%%%%%%%%%%%%%%%%%%%%%%%%%%%%%%%%

\begin{thebibliography}{99}
\bibitem[AN]{AN12}
N.~Abe, H.~Nakaoka, \emph{General heart construction on a triangulated category (II): Associated homological functor}.
Appl. Categ. Structures 20 (2012), no. 2, 161--174.
\bibitem[ARS]{ARS}
M.~Auslander, I.~Reiten, S.~O.~Smal\o, \emph{Representation Theory of Artin Algebras}.
Cambridge Studies in Advanced Mathematics 36, Cambridge University Press,
Cambridge, 1997.
\bibitem[Aus]{Aus66}
M.~Auslander, \emph{Coherent functors}.
1966 Proc. Conf. Categorical Algebra (La Jolla, Calif., 1965) pp. 189--231 Springer, New York.
\bibitem[BBD]{BBD}
A.~Beilinson, J.~Bernstein, P.~Deligne, \emph{Faisceaux Pervers (Perverse sheaves)}.
Analysis and Topology on Singular Spaces, I, Luminy, 1981, Asterisque 100 (1982) 5-171 (in French).
\bibitem[Buh]{Buh10}
T.~B\"uhler, \emph{Exact categories}.
Expo. Math. 28 (2010), no. 1, 1--69.
\bibitem[Che]{Che10}
X-W.~Chen, \emph{Three results on Frobenius categories}.
Math. Z. 270 (2012), no. 1-2, 43--58.
\bibitem[Eno17]{Eno17}
H.~Enomoto, \emph{Classifying exact categories via Wakamatsu tilting}.
J. Algebra 485 (2017), 1--44.
\bibitem[Eno18]{Eno18a}
H.~Enomoto, \emph{Classifications of exact structures and Cohen-Macaulay-finite algebras}.
Adv. Math. 335 (2018), 838--877.
\bibitem[Eno19]{Eno18}
H.~Enomoto, \emph{Relations for Grothendieck groups and representation-finiteness}.
J. Algebra 539 (2019), 152--176.
\bibitem[Fre]{Fre}
P.~Freyd, \emph{Representations in abelian categories}.
1966 Proc. Conf. Categorical Algebra (La Jolla, Calif., 1965) pp. 95--120 Springer, New York.
\bibitem[Gab]{Gab}
P.~Gabriel, \emph{Des cat\'egories ab\'eliennes}.
Bull. Soc. Math. France 90 (1962), 323-448.
\bibitem[GL]{GL91}
W.~Geigle, H.~Lenzing, \emph{Perpendicular categories with applications to representations and sheaves}.
J. Algebra 144 (1991), no. 2, 273--343.
\bibitem[Gr]{Gr57}
A.~Grothendieck, \emph{Sur quelques points d'alg\`ebre homologique}.
T\^ohoku Math. J. (2) 9 (1957) 119-221.
\bibitem[INP]{INP18}
O.~Iyama, H.~Nakaoka, Y.~Palu, \emph{Auslander--Reiten theory in extriangulated categories}.
arXiv:1805.03776v2, May. 2018.
\bibitem[Kel]{Kel90}
B.~Keller, \emph{Chain complexes and stable categories}.
Manuscripta Math. 67 (1990), no. 4, 379-417.
\bibitem[KZ]{KZ08}
S.~Koenig, B.~Zhu, \emph{From triangulated categories to abelian categories: cluster tilting in a general
framework (English summary)}.
Math. Z. 258(1)(2008), 143-160.
\bibitem[Len]{Len98}
H.~Lenzing, \emph{Auslander's work on Artin algebras}.
in Algebras and modules, I (Trondheim,
1996), 83--105, CMS Conf. Proc., 23, Amer. Math. Soc., Providence, RI, 1998.
\bibitem[Liu]{Liu15}
Y.~Liu, \emph{Hearts of cotorsion pairs are functor categories over cohearts}.
arXiv:1504.05271v5, Apr. 2015.
\bibitem[LN]{LN19}
Y.~Liu, H.~Nakaoka, \emph{Hearts of twin cotorsion pairs on extriangulated categories}.
J. Algebra 528 (2019), 96--149.
\bibitem[Nak]{Nak11}
H.~Nakaoka, \emph{General heart construction on a triangulated category (I): Unifying $t$-structures and cluster tilting subcategories}.
Appl. Categ. Structures 19 (2011), no. 6, 879--899.
\bibitem[NP]{NP19}
H.~Nakaoka, Y.~Palu, \emph{Extriangulated categories, Hovey twin cotorsion pairs and model structures}.
Topol. G\'eom. Diff\'er. Cat\'eg, vol LX (2019), Issue 2, 117-193.
\bibitem[Pop]{Pop}
N.~Popescu, \emph{Abelian categories with applications to rings and modules}.
London Mathematical Society Monographs, No. 3. Academic Press, London-New York, 1973. {\rm xii}+467 pp.
\bibitem[ZZ]{ZZ18}
P.~Zhou, B.~Zhu, \emph{Triangulated quotient categories revisited}.
J. Algebra 502 (2018), 196--232.
\end{thebibliography}
\end{document}